\documentclass{article}
\usepackage{amsmath}
\usepackage[utf8]{inputenc}
\usepackage{xcolor}
\usepackage{amsmath}
\usepackage{amsfonts}
\usepackage{amssymb}
\usepackage{bbold}
\usepackage{graphicx}
\usepackage{stmaryrd}
\usepackage{float}
\usepackage{babel}
\usepackage{amsthm}
\usepackage[linesnumbered,ruled,vlined]{algorithm2e}
\usepackage{url}
\usepackage{esint}
\usepackage{hyperref}
\usepackage{enumitem}
\usepackage[export]{adjustbox}
\usepackage[sort, numbers]{natbib}
\usepackage[notcite,notref]{showkeys}
\title{Existence and global behaviour of solutions of a parabolic problem involving the fractional $p$-Laplacian in porous medium}
\author{Loïc Constantin\footnote{\href{loic.constantin@univ-pau.fr}{loic.constantin@univ-pau.fr}}, Jacques Giacomoni\footnote{\href{jacques.giacomoni@univ-pau.fr}{jacques.giacomoni@univ-pau.fr}}, Guillaume Warnault\footnote{\href{guillaume.warnault@univ-pau.fr}{guillaume.warnault@univ-pau.fr}}}
\date{\small{LMAP (UMR E2S UPPA CNRS 5142) B\^atiment IPRA, Avenue de l'Universit\'{e}, 64013 Pau, France}}

\newtheorem{definition}{Definition}[section]
\newtheorem{proposition}[definition]{Proposition}
\newtheorem{corollary}[definition]{Corollary} 
\newtheorem{lemma}[definition]{Lemma}
\newtheorem{theorem}[definition]{Theorem}
\newtheorem{remark}[definition]{Remark}

\newtheorem{propriete}[definition]{Property}

\newcommand{\R}{{\mathbb{R}}}
\newcommand{\Rn}{{\mathbb{R}^d}}
\newcommand{\pfrac}{{(-\Delta)^s_p}}

\newcommand{\Wspz}{{W^{s,p}_0(\Omega)}}

\newcommand{\dt}{{\Delta t}}
\newcommand{\betam}{{\beta}}
\newcommand{\pow}[2]{{\lceil#1\rceil^{#2}}}

\numberwithin{equation}{section}
\renewenvironment{abstract}
{\begin{quote}
\noindent \rule{\linewidth}{.5pt}\par{\bfseries \abstractname.}}
{\medskip\noindent \rule{\linewidth}{.5pt}
\end{quote}
}

\begin{document}

\maketitle

\begin{abstract}
In this paper, we prove the existence and the uniqueness of a {weak and mild} solution of the following nonlinear parabolic problem involving the porous $p$-fractional Laplacian:
 
\begin{equation*}
\begin{cases}
\partial_t u+\pfrac (|u|^{m-1}u)=h(t,x,|u|^{m-1}u) &  \text{in} \; (0,T)\times \Omega,\\
u=0 & \text{in} \;  (0,T) \times \Rn\backslash \Omega, \\
u(0,\cdot)=u_0 &   \text{in} \; \Omega .
\end{cases}\
\end{equation*}
We also study {further} the {the homogeneous} case $h(u)=|u|^{q-1}u$ with $q>0$. In particular we investigate global time existence, uniqueness, global behaviour of weak solutions and stabilization.

\end{abstract}

\section{Introduction}
{Let} $p\in (1,\infty),$ $q>0,$ $s\in (0,1)$, $m>1$ and $Q_T:=(0,T)\times \Omega$ where $\Omega$ is a bounded domain of $\Rn$ with smooth boundary. The fractional $p$-Laplacian  operator on the Sobolev space $\Wspz$, up to a normalizing constant, states as:

$$\pfrac u(x)= 2\,p.v. \int_\Rn \frac{\pow{u(x)-u(y)}{p-1}}{|x-y|^{d+sp}}dy,$$

where $p.v.$ stands for Cauchy's principal value and we use the notation $\pow{X}{r}:=|X|^{r-1}X$.\\ In this paper, we study

\begin{equation}\label{pborig}
\begin{cases}\tag{$E$}
\partial_t u+\pfrac \pow{u}{m}={h}(t,x,\pow{u}{m})&  \text{in} \;Q_T,\\
u=0 &  \text{in} \;   (0,T)\times \Rn\backslash \Omega, \\
u(0,\cdot)=u_0&  \text{in} \; \Omega .
\end{cases}\
\end{equation}

Defining $v=\pow{u}{m}$, we investigate the auxiliary problem \eqref{pbh} satisfied by $v$ (see Subsection \ref{subsection_results}) and then prove existence, uniqueness of a weak-mild solution (according to Definition \ref{wmildsol}) to \eqref{pbh}. We also focus on the global behaviour of nonnegative weak-mild solutions for the power type case $h(u)=\pow{u}{q}$.

The study of nonlocal operators equations has been of great interest in recent times due to the occurrence in various physical phenomena where long-range type interaction takes place, motivating the development of an inherent mathematical theory see {\it e.g.} \cite{barbu2, pme, levine, lady}.  For instance the elliptic and parabolic type equations involving these operators appear in various fields such as finance, stochastic processes of Lévy type, physics, population dynamics, fluid dynamics.

In particular for $p=2$, the nonlocal diffusion evolution type problems have been studied using the {autonomous} model 

\begin{equation*}
\partial_t u +(-\Delta)^s(\pow{u}{m})=0.
\end{equation*}

For $m<1$ this equation is called the fractional fast diffusion equation, while the case $m>1$ is referred as the fractional porous media equation. This model has been vastly studied, we can see {\it e.g.} \cite{fpme, gfpme, bonforte, extension} where various methods (Caffareli-Silvestre extension technique, implicit time discretization and monotone operator methods) were used to solve the problem in the all space as well as in a bounded domain with suitable boundary conditions.

Diffusion case $m=1$ entails the following more general fractional $p$-Laplacian  model

\begin{equation}\label{m=1 case}
\partial_t u +(-\Delta)^s_pu=0.
\end{equation}

Equation {\eqref{m=1 case} has been deeply studied in different former contributions, see {\it e.g.} \cite{fpl, vazq, abdellaoui}. The elliptic associated problem has also been issued for existence, H\"older regularity and properties of weak solutions, see {\it e.g.} \cite{delpezzo-quaas, reghold}. The corresponding double nonlinear operator has recently met an increase of interest in the literature. In particular the porous fractional $p$-Laplacian  evolution equations are studied in \cite{kato, exi, CollierHauer}}.

Taking into account these contributions on the subject, we further study \eqref{pborig} in the present article. First, we enlarge the investigation of the case $m>1$ according to \cite{exi} where a particular doubly evolution equation is studied for $m\in [\frac{1}{2p-1},1)$. Additionaly, we study the nonautonomous version of \eqref{m=1 case} with additional source terms  compared to \cite{kato}.  
In \cite{CollierHauer}, the authors study the mild solution of the evolution problem involving the operator $\mathcal O: u\mapsto \pfrac\phi(u)+f(\cdot,u)$ (that includes the case $\phi(t)=\pow{t}{m}$). They restrict the perturbation of the diffusion operator to $f$ a global Lipschitz function in order to get $\mathcal O$ as a $\omega$-quasi $m$-$T$-accretive operator. In our work, considering a more general type of perturbations, we also use the accretive operators theory for proving existence of weak solutions. By combining the notions of weak and mild solutions, we further prove uniqueness of solutions which has not been established so far for this class of nonlinear and nonlocal evolution equations. We precisely obtain a bounded weak-mild solution verifying a variational formulation as well as a $L^1$-contraction property and an energy inequality. The uniqueness of the general problem for $h$ is obtained via Gronwall's Lemma and requires local Lipschitz regularity of $h$. To perform this weak-mild solutions approach, we first prove that, for $A=\pfrac \pow{\cdot}{m}$, the domain is dense in $L^1$ and $A$ is accretive in $L^1$ (which is up to our knowledge the only accretivity in $L^p$ spaces we can expect in case $m>1$). Then, we further investigate the global behaviour of the considered class of solutions. More precisely, we prove blow-up and finite time extinction behaviour under appropriate assumptions on parameters and the initial data. Using the comparison principle and appealing the semigroup theory, we also prove stabilization results, namely asymptotic convergence to a nontrivial steady state solution as $t$ goes to $\infty$. These results concerning qualitative properties of weak solutions are completely new with respect to the current literature.

\subsection{Functional setting and preliminaries}

We first present the functional settings, for additional references we refer for instance \cite{ bisci, rando}. The space framework is the fractional Sobolev space $W^{s,p}(\R^d)$ defined as:

$$W^{s,p}(\Rn)=\bigg\{ u\in L^p(\Rn) \; | \; \int_\Rn \int_\Rn  \frac{|u(x)-u(y)|^p}{|x-y|^{d+sp}}dy dx<\infty \bigg\},$$

endowed with the natural norm:

$$\|u\|_{W^{s,p}(\Rn)}=\bigg (\int_\Rn |u|^pdx +  \int_\Rn \int_\Rn  \frac{|u(x)-u(y)|^p}{|x-y|^{d+sp}}dy dx \bigg)^\frac{1}{p}.$$

The space $\Wspz$ is defined by $\Wspz= \big \{ u\in W^{s,p}(\Rn) \; | \; u=0 \text{ on } \Rn \backslash \Omega \big \},$ endowed with the Banach norm

$$\|u\|_\Wspz = \left( \int_\Rn \int_\Rn  \frac{|u(x)-u(y)|^p}{|x-y|^{d+sp}}dy dx\right)^{\frac1p}.  $$

The space $\Wspz$ is a reflexive space and a Poincar\'e inequality (see {\it e.g.} \cite[Th. 6.5]{bisci}) provides the equivalence of the two norms $\|\cdot \|_{W^{s,p}(\Rn)}$ and $\|\cdot\|_\Wspz$ for this space.\\
For $sp<d$ and $r\in [1,\frac{dp}{d-sp}]$, we have the continuous embedding $\Wspz\hookrightarrow L^r(\Omega)$ and if $r<\frac{dp}{d-sp}$ the embedding is compact (see {\it e.g.} \cite[Coro. 7.2]{rando}). For $sp\geq d$, we have the compact embedding $\Wspz\hookrightarrow L^r(\Omega)$ for all $r\in [1,+\infty)$.\\

By the definition of the fractional $p$-Laplacian, we have for any $u,v \in \Wspz$:

$$\langle\pfrac u,v\rangle=\int_\Rn \int_\Rn \frac{\pow{u(x)-u(y)}{p-1}(v(x)-v(y))}{|x-y|^{d+sp}}dx dy.$$
We introduce 
$$X_T:=\{ v\in  L^\infty(Q_T)\cap L^\infty(0,T ,\Wspz)\; | \; \partial_t  v \in L^2(Q_T) \}$$
and the Sobolev-Bochner space as follows:
\begin{equation*}
    \begin{split}
 \mathcal W_T&:=W^{1,m+1,(m+1)' }(0,T,L^{m+1}(\Omega),L^{(m+1)'}(\Omega))\\
 &=\{u\in L^{m+1}(0,T,L^{m+1}(\Omega)) \; | \; u'\in L^{(m+1)'}(0,T,L^{(m+1)'}(\Omega))\}.
 \end{split}
 \end{equation*}
where $(m+1)'$ is the conjugate exponent of $m+1$. Note that Lemma \ref{ipp} infers $\mathcal{W}_T \hookrightarrow C([0,T];L^2(\Omega))$ and Aubin-Simon Theorem paired with the interpolation inequality gives $X_T \subset C([0,T];L^r(\Omega))$ for all $r\in [1,\infty).$

\subsection{Main results}\label{subsection_results}

We consider the {auxiliary} problem associated to $(E)$:

\begin{equation}\label{pbh}\tag{$P_h$}
    \begin{cases}
     \partial_t \beta(v)+\pfrac v=h(t,x,v) &  \text{in} \;Q_T,\\
      v=0 &  \text{in} \; (0,T) \times \Rn\backslash \Omega, \\
      v(0,\cdot)=v_0 &  \text{in} \; \Omega ,
    \end{cases}\
\end{equation}

where $\beta: \R\to \R$ is defined by $\beta(\theta):=\pow{\theta}{\frac{1}{m}}$ and $h:\mathbb{R}^+\times \Omega \times \mathbb{R} \to \mathbb{R}$ such that for any $(t,x)\in \mathbb R^+\times \Omega$, the function $\theta \mapsto h(t,x,\theta )$ is continuous.\\
For some $q>0$, we define the following hypothesis:
\begin{equation}\label{f1}\tag{$H_q$}
\sup_{Q_\infty}|h(\cdot,\cdot,\theta )|\leq C_h (1+|\theta|^\frac{q}{m})  \text{ for any } \theta\in \R\end{equation}
where $Q_\infty:=(0,+\infty)\times \Omega.$
We also require the following hypothesis:\\ For any $T>0$ and for any compact set $K\subset \mathbb R$, there exists $C=C(K,T) \geq 0$  such that  for any $\theta_1$, $\theta_2\in K$
\begin{equation}\label{f3}\tag{$H$}
\sup_{Q_T}|h(\cdot,\cdot,\theta_1)-h(\cdot,\cdot,\theta_2)|\leq C|\beta(\theta_1)-\beta(\theta_2)|.
\end{equation}
{A prototype example of $h$ satisfying \eqref{f1} and \eqref{f3} is $h(\cdot,\cdot,\theta)=\pow{\theta}{\frac{q}{m}}$ with $q\geq 1$. Considering initial data $u_0$, $v_0$ in $L^\infty(\Omega)$, we now define the different required notions of solutions to \eqref{pbh} as follows}:
\begin{definition}[Weak solution]\label{defsol2}A function $v \in X_T$ is a weak solution of \eqref{pbh} if $\beta(v)\in C(0,T,L^{2}(\Omega))\cap L^{m+1}(Q_T)$ and $v(0,\cdot)=v_0$ {\it a.e.} in $\Omega$ such that for any $t\in [0, T]$:

\begin{equation}\label{fv2}
\begin{split}
\bigg[\int_\Omega \beta(v) \phi\,dx\bigg]^t_0-\int_0^t\int_\Omega \beta(v)\partial_t \phi \,dxd\tau+&\int_0^t\langle\pfrac v,\phi\rangle\,d\tau\\
&=\int_0^t\int_{\Omega}h(\tau,x,v)\phi\,dxd\tau,
\end{split}
\end{equation}
for any $\phi \in \mathcal{W}_T\cap L^p(0,T;\Wspz)$. \end{definition}and 
\begin{definition}[Strong solution]\label{defsol1}
A function $v\in  X_T$ is a strong solution of \eqref{pbh} if $v(0,\cdot)=v_0$ {\it a.e.} in $\Omega$, $\partial_t\beta(v)\in L^{1}(Q_T)$, $\beta(v)\in C([0,T],L^1(\Omega))$ such that for any $t\in [0, T]$:
\begin{equation*}
\int_0^t\int_\Omega \partial_t \beta(v) \phi\,dxd\tau+\int_0^t\langle\pfrac v ,\phi\rangle\,d\tau=\int_0^t\int_\Omega h(\tau,x,v)\phi\,dxd\tau,
\end{equation*}
for any $\phi \in L^{\infty}(Q_T)\cap L^p(0,T;\Wspz)$ such that $\partial_t \phi \in L^{(m+1)'}(Q_T)$. 
\end{definition}

Similarly to  \citep[Def 4.3, p.130]{barbu2}, we define the notion of mild solutions of the following problem

\begin{equation*}\label{pbo}\tag{$P_0$}
    \begin{cases}
     \partial_t u+\pfrac \pow{u}{m}=f &  \text{in} \;Q_T,\\
      u=0 & \text{in} \;(0,T) \times \Rn\backslash \Omega, \\
      u(0,\cdot)=u_0&  \text{in} \; \Omega .
    \end{cases}\
\end{equation*}

For this, we recall 
 in Appendix B the notion of approximate solution with Definition \ref{defapp}. Hence we have
\begin{definition}[Mild solutions]\label{mildsol}\ \\
Let $T>0$ and let $f\in L^1(Q_T)$. A mild solution of \eqref{pbo} is a function $u$ belonging to $C([0,T];L^1(\Omega))$ satisfying for any $\epsilon>0$, there is an $\epsilon$-approximate solution $U$  such that $\|U-u\|_{L^1(\Omega)}<\epsilon$ for all $t\in[0,T]$.\\
In a similar way, a mild solution of problem \eqref{pborig} is a function $u\in  C([0,T];L^1(\Omega))$ which is a mild solution of \eqref{pbo} with $f=h(\cdot,\cdot,\pow{u}{m}).$
\end{definition}
\begin{definition}[Weak-mild solutions]\label{wmildsol}
Let $T>0$. A weak-mild solution $v$ of \eqref{pbh} is a weak solution such that $\beta(v)$ is a mild solution of \eqref{pborig}.\\
We call $T$-weak-mild solution a function $v\in L^\infty_{loc}(0,T;L^\infty(\Omega))$ such that $v$ is a weak-mild solution of \eqref{pbh} on $Q_{\tilde T}$ for any $\tilde T<T$.
\end{definition}
\begin{remark}
Let $v$ be a weak-mild solution of \eqref{pbh} on $Q_T$ then $v$ is a $T$-weak-mild solution of \eqref{pbh}.
\end{remark}

We define the domain of the operator $A: u \mapsto \pfrac \pow{u}{m}$:

$$D(A)=\{ u  \; | \; \pow{u}{m}\in \Wspz, \; \pfrac \pow{u}{m}\in L^1(\Omega) \}.$$

We start with the following result proved in Appendix \ref{accsec}.

\begin{proposition}\label{propdom}
Let $v_0\in \Wspz\cap L^\infty(\Omega)$ then  $\beta(v_0)\in \overline{D(A)}^{L^1(\Omega)}$. 
\end{proposition}

We derive easily the following density result:
\begin{corollary}
$D(A)$ is dense in $L^1(\Omega)$.
\end{corollary}

\begin{proof}
Let $\phi\in C^\infty_0(\Omega)$. We have $\pow{\phi}{m} \in L^\infty(\Omega) $ and $m>1$ gives:

$$|\pow{\phi}{m}(x)-\pow{\phi}{m}(y)|\leq C|\phi(x)-\phi(y)|,$$

and so $\pow{\phi}{m}\in \Wspz$. Using Proposition \ref{propdom}, we obtain $ C^\infty_0(\Omega)\subset \overline{ D(A)}^{L^1(\Omega)}$ from which one has $\overline{D(A)}^{L^1(\Omega)}=L^1(\Omega)$.
\end{proof}

We now state our main results:

\begin{theorem}\label{thgen}
Let $v_0\in L^\infty (\Omega)\cap \Wspz$. Assume \eqref{f1} for some $q>0$. Then, there exists $T\in (0,+\infty]$ such that  \eqref{pbh} admits a $T$-weak-mild solution $v$. Moreover, we have:
\begin{enumerate}[label=(\roman*)]
    \item if $q\in (0,1]$, then $T=\infty$.
\item Let $T'\in (0,T)$ and let $u$ be a weak-mild solution of \eqref{pbh} on $Q_{T'}$ with the right-hand side data $\hat h$ and the initial data $u_0\in { \Wspz\cap L^{\infty}(\Omega)}$,
then
\end{enumerate}
\begin{equation*}
\sup_{[0,T']}\|\beta(u)-\beta(v)\|_{L^1(\Omega))}\leq \|\beta(u_0)-\beta(v_0)\|_{L^1(\Omega)}+\int_0^{T'}\|\hat h(t,\cdot,u)-h(t,\cdot,v)\|_{L^1(\Omega)}\,dt.
\end{equation*}
\begin{enumerate}[label=(\roman*)]\setcounter{enumi}{2}
   \item If \eqref{f3} holds then the $T$-weak-mild solution is unique.
    \item If  $\beta(v_0) \in D(A)$ and let $\tilde T<T$ such that $h(\cdot,\cdot,v)\in W^{1,1}([0,\tilde T];L^1(\Omega))$ then 
    $v$ is a strong solution on $Q_{\tilde T}$.
\end{enumerate}
\end{theorem}

Under additional assumptions on $h$ and $v_0$, we obtain global existence results. Precisely, defining $d(\cdot,\partial \Omega)$ the distance function up to the boundary, we have:
\begin{theorem}\label{subhomo2}
Under the hypothesis in Theorem \ref{thgen}, assume that $p>\frac{q}{m}+1$ and there exists $C>0$ such that {$0\leq v_0\leq C d(\cdot,\partial \Omega)^s$}. Then, there exists a nonnegative $\infty$-weak-mild solution  $v_{glob}\in L^\infty(Q_\infty)$ of \eqref{pbh}.
\end{theorem}
{In the case where \eqref{f3} is satisfied, the following result giving the existence of maximal solutions holds:}
\begin{theorem}\label{GB2}
Let $T\in (0,+\infty)$ and $v$ be the $T$-weak-mild solution of \eqref{pbh} for the data $v_0\in L^\infty(\Omega)\cap\Wspz$ and $h$ satisfying \eqref{f1} and \eqref{f3}. {Then, we have the alternative:
\begin{itemize}
    \item [-] Either $\lim_{t\to T}\|v(t)\|_{L^\infty(\Omega)} = \infty$ { ($v$ is a maximal solution).}
\item[-] Or there exists $\varepsilon>0$ {\it s.t.} $v$ is a weak-mild solution on $Q_{T+\varepsilon}$ ($v$ is extendable on $(0,T+\varepsilon)$).
\end{itemize}}
\end{theorem}
{\begin{remark}\label{tmax}
Under the condition \eqref{f3}, the unique $T$-weak-mild solution v of \eqref{pbh} obtained by Theorem \ref{thgen}  can be extended on $(0,T') $ with $T'\geq T$ such that $\displaystyle\lim_{t\to T'}\|v(t)\|_{L^\infty(\Omega)}=\infty$ if $T'<\infty$. Thus we define the maximum life time of the solution 
$$T_{max}=\sup \{ T>0 \; | \;  v \text{ is the weak-mild solution of \eqref{pbh} on }Q_T \}$$
and $v$ is the $T_{max}$-weak-mild solution of \eqref{pbh}.
\end{remark}}

Next, we focus on the following equation:
\begin{equation}\tag{$P_q$}\label{pbq}
    \begin{cases}
     \partial_t \beta(v)+\pfrac v=\pow{v}{\frac{q}{m}} & \text{in} \;Q_T,\\
      v=0 & \text{in} \; (0,T) \times \Rn\backslash \Omega, \\
      v(0,\cdot)=v_0 & \text{in} \; \Omega .
    \end{cases}\
\end{equation}

In particular, we prove the following global existence and stabilization result:

\begin{theorem}\label{subhomo1}

Let $p>\frac{q}{m}+1$ and $v_0\in L^\infty(\Omega)\cap \Wspz $ be such that $\frac1C d(\cdot,\partial \Omega)^s \leq v_0\leq C d(\cdot,\partial \Omega)^s$ for some $C>0$. Then, there exists $v_{glob}\in L^\infty(Q_\infty)$ a nonnegative $\infty$-weak-mild solution of \eqref{pbq} such that, for any $r<\infty$, $v_{glob}$ tends to $ v_\infty$ in $L^r(\Omega)$ as $t$ goes to $\infty$ where $v_\infty$ is the unique nontrivial nonnegative stationary solution.
\end{theorem}

Finally the theorems below state the conditions to get quenching and blow-up in finite time respectively:

\begin{theorem}\label{GB1}Let $v_0\in L^\infty(\Omega)\cap \Wspz $, $q\leq 1$ and $p<\frac{q}{m}+1$. Let $v$ be a $\infty$-weak-mild solution of \eqref{pbq}.
Then,  assume that $\|v_0\|_{L^{r+\frac{1}{m}}(\Omega)}$ is small enough where $r$ is a suitable parameter, the extinction in finite time of $v$ holds, that is there exists $ t_0$ such that $v(t)=0$ for all $t\geq t_0$.



\end{theorem}

Note that if {$q\geq 1$}, then \eqref{f3} holds and $T_{max}$ is well-defined by Remark  \ref{tmax}.

\begin{theorem}\label{GB3}
Let $p<\frac{q}{m}+1$ and $v_0\in L^\infty(\Omega)\cap\Wspz$ satisfying  
$$E(v_0):=\frac{1}{p}\|v_0\|_\Wspz^p-\frac{m}{m+q}\|v_0\|^{\frac{q}{m}+1}_{L^{{\frac{q}{m}+1}}(\Omega)}\leq 0.$$
{Then, the following assertions hold}:
\begin{itemize}
\item If $q>1$ then $T_{max}<\infty$ and $v$ the unique $T_{max}$-weak-mild solution of \eqref{pbh} satisfies 
$$\lim_{t\to T_{max}} \|v(t)\|_{L^{\infty}(\Omega)}=+\infty.$$
    \item  If $q=1$, then $\infty$-weak-mild solution $v$ satisfies $$\lim_{t\to \infty} \|v(t)\|_{L^{{\frac{1}{m}+1}}(\Omega)}=+\infty.$$
    \item If $q<1$, then $\infty$-weak-mild solution $v$ satisfies $$\lim_{T\to \infty} \|v\|_{L^\infty(Q_T)}=+\infty.$$
\end{itemize}
\end{theorem}

The rest of this paper is organized as follows. In Section \ref{partelli}, we study the elliptic problem associated with \eqref{pbh}. Section \ref{partiepara} is dedicated to the proofs of Theorems \ref{thgen}, \ref{subhomo2} and \ref{GB2}. In Section \ref{mainpart}, we study the problem \eqref{pbq} and establish Theorem \ref{subhomo1} in Subsection \ref{partsubho} and  Theorem \ref{GB1} and \ref{GB3} in Subsection \ref{partGB}.

\section{Elliptic problem}\label{partelli}

In this section, we investigate:
\begin{equation} \tag{Q}\label{pbe}
    \begin{cases}
      \beta(v)+\lambda \pfrac v=f & \text{in} \;\Omega,\\
      v=0 & \text{in} \; \Rn\backslash \Omega,
    \end{cases}\
\end{equation}
where $\lambda$ is a positive parameter. 
\begin{definition}\label{solelli}
Let $f\in L^{\infty}(\Omega)$. A weak solution of \eqref{pbe} is a function $v$ belonging to $ \Wspz\cap L^{\frac{1}{m}+1}(\Omega)$ such that:

\begin{equation}\label{fve}
\int_\Omega \beta(v)\phi\,dx+\lambda
\langle \pfrac v,\phi\rangle=\int_\Omega f \phi\,dx,
\end{equation}
for any $\phi \in \Wspz\cap L^{\frac{1}{m}+1}(\Omega)$.
\end{definition}
We recall the following weak comparison principle.
\begin{proposition}[Comparison principle]\label{princcomp}
Let $\lambda>0$ and $r \in [1,\infty]$. Let $g:\mathbb R\to \mathbb R$ be a nondecreasing function and $u,v\in L^r(\Omega) \cap \Wspz$  such that $g(u)$, $g(v)$ belong to $L^{r'}(\Omega)$ and satisfy
$$\int_\Omega g(u) \phi\,dx+ \lambda\langle\pfrac u, \phi\rangle \leq \int_\Omega g(v) \phi\,dx+ \lambda\langle\pfrac v,\phi\rangle ,$$
for any nonnegative $\phi \in L^r(\Omega)\cap \Wspz $. Then, $u\leq v$ {\it a.e.} in $\Omega$.
 \end{proposition}
\begin{proof}
Denoting {$\Omega^+=\mathrm{Supp}(u-v)^+$} and choosing $\phi=(u-v)^+ \in L^r(\Omega) \cap \Wspz$ as a test function, we get
$$\int_\Omega (g(u)-g(v))(u-v)^+\,dx+\lambda \langle\pfrac u -\pfrac v, (u-v)^+\rangle\leq 0.$$
Since $g$ is nondecreasing, the first term in the left-hand side is nonnegative, hence  we have:
\begin{equation}\label{23}
\langle\pfrac u -\pfrac v, (u-v)^+\rangle\leq 0.
\end{equation}
{By straightforward computations, we have
\begin{equation*}
\begin{split}
&\int_{\mathbb{R}^{2d}}\big (\pow{u(x)-u(y)}{p-1}-\pow{v(x)-v(y)}{p-1}\big) \big ( (u-v)^+(x)-(u-v)^+(y)\big )dxdy\geq \nonumber\\ 
&\int_{\Omega^+\times\Omega^+}\big (\pow{(u(x)-u(y))}{p-1}-\pow{(v(x)-v(y)}{p-1}\big) \big ( (u-v)(x)-(u-v)(y)\big )dxdy.
\end{split}
\end{equation*}
From above inequality and \eqref{23},} we deduce $u\leq v$ {\it a.e.} in $\Omega$.
\end{proof}
\begin{theorem}\label{exielli}
Let $f\in L^{\infty}(\Omega)$. Then, \eqref{pbe} admits a unique weak solution (in sense of Definition \ref{solelli}) belonging to $ C^s(\Rn)$. Moreover, if $f$ is nonnegative then the solution is nonnegative.
\end{theorem}

\begin{proof}
We consider the functional defined on $\Wspz\cap L^{\frac{1}{m}+1}(\Omega)$ by:
$$J(v)=\frac{\lambda}{p}\|v\|^p_\Wspz + \frac{m}{m+1}\|v\|^{\frac{1}{m}+1}_{L^{\frac{1}{m}+1}(\Omega)}-\int_\Omega fv\,dx.$$
Thus, applying H\"older inequality, we have for suitable constants $c$ and $c'$ :
$$J(v) \geq c \|v\|^p_\Wspz+c'\|v\|^{\frac{1}{m}+1}_{L^{\frac{1}{m}+1}(\Omega)}-\|f\|_{L^{m+1}}\|v\|_{L^{\frac{1}{m}+1}(\Omega)}.$$
Hence $J$ is coercive.\\
Let {$(v_n)_{n\in\mathbb{N}}$} be a minimizing sequence associated to the infimum of $J$
, then $(v_n)$ is bounded in the reflexive spaces $\Wspz$ and $L^{\frac{1}{m}+1}(\Omega)$. Therefore, there exists $v\in\Wspz\cap L^{\frac{1}{m}+1}(\Omega)$ such that up to a subsequence, $v_n \rightharpoonup v$ in $\Wspz$ and $L^{\frac{1}{m}+1}(\Omega)$.\\
From weakly lower semicontinuity ({\it w.l.s.c.} for short) of the associated norms, $v$ is a global minimizer of $J$ and since $J$ is G\^ateaux-differentiable, $v$ satisfies \eqref{fve} for any $\phi \in \Wspz\cap L^{\frac{1}{m}+1}(\Omega)$.\\
Now, we show that $v\in L^\infty(\Omega)$. For this, we consider $w\in \Wspz$ the unique nonnegative solution to	%
\begin{equation*} 
    \left\{\begin{array}{ll}
         \lambda \pfrac w=\|f\|_{L^\infty} &\text{in} \;\Omega,\\
      w=0  &\text{in} \; \Rn\backslash \Omega.
    \end{array}\right.
    \end{equation*}
By \cite[Th. 2.7]{reghold}, $w\in L^\infty(\Omega)$. Since the function $\beta$ is nondecreasing and $\beta(0)=0$, Proposition \ref{princcomp} yields that $-w \leq v \leq w$
and  thus $v\in L^\infty(\Omega)$. Finally \cite[Th. 2.7]{reghold} gives $v\in C^s(\Rn)$. 

The uniqueness of the solution follows from Proposition \ref{princcomp}. Moreover if $f\geq0$, applying once again Proposition \ref{princcomp} with $u\equiv 0$, we deduce that $v\geq0$.
\end{proof}
\section{Quasilinear parabolic problem}\label{partiepara}
\subsection{Existence}\label{subsection 3.1}
We consider the auxiliary problem:  
\begin{equation}\label{pbpsi}\tag{$P_{h,R}$}
    \begin{cases}
     \partial_t \beta(v)+\pfrac v=h_R(t,x,v) & \text{in} \;Q_T,\\
      v=0 & \text{in} \;  (0,T) \times \Rn\backslash \Omega, \\
      v(0,\cdot)=v_0 & \text{in} \; \Omega ,
    \end{cases}\
\end{equation}
where $h_R$ is defined, for some $R>0$,  by $h_R(t,x,\theta)=h(t,x,sgn(\theta)\min(|\theta|,R))$ where we recall the definition of the $sgn$ function:
$$sgn(x)= \left\{ \begin{array}{l l}
     \frac{x}{|x|}& \mbox{if } x\neq 0  \\
      0. &\mbox {else} 
\end{array}\right. $$
First, we prove the following theorem:

\begin{theorem}\label{thexi1}
Let $T>0$ and $v_0\in L^\infty(\Omega) \cap \Wspz$. Assume \eqref{f1}, then \eqref{pbpsi} admits a weak solution $v\in X_T$ in the sense of Definition \ref{defsol2}.
\end{theorem}

\begin{proof}
We use a time-discretization method and we divide the proof in several steps.\\
\textbf{\underline{Step 1:}} Time-discretization scheme\\
Let $N\in \mathbb{N}^*$, we set $\dt=\frac{T}{N}$, $t_n=n\dt$, {for $1\leq n\leq N$}.\\
We define the sequence $(v^n)_{n} \subset L^{\frac{1}{m}+1}(\Omega)\cap \Wspz$ with $v^0=v_0$ and by the iterative scheme:
\begin{equation} \tag{$P_n$} \label{pbn}
    \begin{cases}
      \frac{1}{\dt}\beta(v^n)+ \pfrac v^n=h^n+\frac{1}{\dt}\betam(v^{n-1}) & \text{in} \;\Omega,\\
      v^n=0 & \text{in} \; \Rn \backslash \Omega,
    \end{cases}\
\end{equation}
where for any $x\in \Omega$
$$h^n(x):=\fint^{t_n}_{t_{n-1}} h_R(\tau ,x,v^{n-1}(x)) d\tau.$$

By Theorem \ref{exielli}, the sequence $(v^n)_{n}$  is well defined since $v_0$ and $h^n$ belong to $ L^\infty(\Omega)$. Furthermore, for any $n$, $v^n\in L^\infty(\Omega)$.\\
We define, for any $n=1,..., N$ on $[t_{n-1},t_n),$
$$h_\dt=h^n,\  v_\dt =v^n, \ \widetilde{ v}  _\dt=\dfrac{t-t_{n-1}}{\dt}(v^n-v^{n-1})+v^{n-1}$$
and
$$  \widetilde{ \beta (v)}  _\dt=\dfrac{t-t_{n-1}}{\dt}(\betam(v^n)-\betam(v^{n-1}))+\betam(v^{n-1}).$$

Then, we get $\partial_t \widetilde{\betam(v)}_\dt+\pfrac(v_\dt)=h_\dt,$ for {\it a.e.} $t\in (0,T)$, in sense of Definition \ref{solelli}.\\

\textbf{\underline{Step 2:}} $L^\infty$-bound

First, we prove that $(v_\dt)$ is uniformly bounded in $L^\infty(Q_T)$.\\
Since $v^n\in L^\infty(\Omega)$ and \eqref{ineg1} holds,  choosing $\phi=\pow{\beta(v^n)}{r-1}\in\Wspz$, for some $r\geq m$, as a test function in the weak formulation of Problem \eqref{pbn}, we obtain:
$$\frac{1}{\dt}\int_\Omega  (\betam(v^{n})-\betam(v^{n-1}))\phi\,dx+\langle(-\Delta)^s_pv^n,\phi\rangle\leq \|h^{n}\|_{L^{r}(\Omega)}\|\phi\|_{L^{\frac{r}{r-1}}(\Omega)}.$$
Since $\langle(-\Delta)^s_pv^n,\phi\rangle \geq 0$,  from above inequality and Hölder inequality, we get:

$$\|\beta(v^n)\|^r_{L^r(\Omega)}\leq  (\dt\|h^{n}\|_{L^{r}(\Omega)}+\|\beta(v^{n-1})\|_{L^{r}(\Omega)})\|\pow{\beta(v^n)}{r-1}\|_{L^{\frac{r}{r-1}}(\Omega)}.$$
Hence by \eqref{f1}, it follows
\begin{equation*}
\begin{split}
\|\beta(v^n)\|_{L^r(\Omega)}& \leq \dt\|h^{n}\|_{L^{r}(\Omega)}+\|\beta(v^{n-1})\|_{L^{r}(\Omega)}\leq\|\beta(v^{0})\|_{L^{r}}+\sum^n_{k=1} \dt\|h^{k}\|_{L^{r}(\Omega)},\\
&\leq \|\beta(v^{0})\|_{L^{r}(\Omega)}+\int_0^T\|h_\dt\|_{L^{r}(\Omega)},
\end{split}
\end{equation*}
%


passing to the limit as $r$ goes to $\infty$, we have that $(v_\dt)_\dt$ is bounded in $L^\infty(Q_T)$ uniformly in $\dt$. More precisely, we obtain
\begin{equation}\label{bdd}
\begin{split}
    \|\beta(v_\dt)\|_{L^\infty(Q_T)}& \leq \|\beta(v_0)\|_{L^\infty(\Omega)}+\int_0^T ||h_\dt ||_{L^\infty(\Omega)},\\
    &\leq \|\beta(v_0)\|_{L^\infty(\Omega)}+C_hT(1+R^{\frac{q}{m}}).
\end{split}
\end{equation}
\textbf{\underline{Step 3:}} A priori estimates

From the convexity of the mapping $v\mapsto \frac{1}{p}\|v\|_\Wspz^p$,  taking the test function $\phi = v^{n}-v^{n-1} \in  \Wspz\cap L^{\frac{1}{m}+1}(\Omega)$ {in the variational formulation of} problem \eqref{pbn}, we get
\begin{equation*}\label{3}
\begin{split}
\frac{1}{\dt}\int_\Omega (\beta(v^n)-\beta(v^{n-1}))(v^n-v^{n-1})\,&dx+\frac{1}{p}\|v^{n}\|_\Wspz^p\\
&\leq  \int_\Omega (v^n-v^{n-1})h^n\,dx+\frac{1}{p} \|v^{n-1}\|_\Wspz^p.
\end{split}
\end{equation*}
%
%
From Young inequality and since the mapping $x\mapsto \pow{x}{m}$ is locally Lipschitz and $(v^n)_n$ is uniformly bounded in $L^\infty(\Omega)$, it follows with some suitable constants $c_1$ and $c_2$
\begin{equation*}\label{12}
\begin{split}
\int_\Omega (v^{n}-v^{n-1}) h^{n}\,dx&\leq c_1 \dt\|h^{n}\|_{L^{2}}^2 +\frac{c_2}{\dt}\|v^{n}-v^{n-1}\|_{L^{2}}^2,\\
&\leq c_1\dt\|h^{n}\|_{L^{2}}^2 +
\frac{1}{2\dt} \int_\Omega (\betam(v^{n})-\betam(v^{n-1}))(v^{n}-v^{n-1})\,dx .
\end{split}
\end{equation*}
Combining the two previous inequalities  and summing from $n=1$ to $n=N'$ with $1\leq N'\leq N$, we get:
\begin{equation}\label{estiDis}
\begin{split}
\|v^{N'}\|^p_\Wspz+  & \frac{1}{\dt}\sum^{N'}_{n=1}  \int_\Omega (\betam(v^{n})-\betam(v^{n-1}))(v^{n}-v^{n-1})\,dx, \\
& \leq C(\|v_0\|_\Wspz^p+\sum^{N'}_{n=1}\dt \|h^n_R\|_{L^2(\Omega)}^2)\leq C(v_0,R,T),
\end{split}
\end{equation}
where the constant is independent of $\dt$. Hence we deduce from \eqref{estiDis} that the sequences $(v_\dt)_\dt$,  $(\widetilde{v}_\dt)_\dt$ are bounded in $L^\infty(0,T , \Wspz)$ and
$$\displaystyle \int_{0}^{T} \int_\Omega \left(\frac{v^{n}-v^{n-1}}{\dt}\right)^2\,dx dt\leq  \frac{c}{\dt}\sum^{N}_{n=1}  \int_\Omega (\betam(v^{n})-\betam(v^{n-1}))(v^{n}-v^{n-1})\,dx  \leq C$$
which gives that $ (\partial_t \widetilde{v}_\dt)_\dt$ is bounded in $L^2(Q_T)$ and 
\begin{equation}\label{11}
\sup_{[0,T]}\|\tilde{v}_\dt- v_\dt\|_{L^2(\Omega)}^2 \leq C\dt.
\end{equation}
\textbf{\underline{Step 4:}} Convergences as $\dt$ goes to $0$.\\
From \textbf{Step 3}, we have, up to a subsequence, $ v_\dt \overset{*}{\rightharpoonup} v_1$ and  $  \widetilde{v}_\dt \overset{*}{\rightharpoonup} v_2$ in $L^\infty (0,T, W_0^{s,p}(\Omega))$ as $\dt$ goes to $0$ and from \eqref{11}, we deduce $v:=v_1=v_2$.\\
Moreover the facts that $(\widetilde{v}_\dt)_\dt$ is bounded in $L^\infty(0,T;\Wspz)$ and $(\partial_t\widetilde{v}_\dt)_\dt$ is bounded in $L^2(Q_T)$ imply together with the fractional Sobolev embedding and the Aubin-Simon Theorem, up to a subsequence, $\tilde{v}_\dt$ tends to $v$ in $C([0,T],L^r(\Omega))$ for any  $r\in[1, \min(2, \frac{dp}{d-sp}))$.\\
Let $\tilde r\geq \min(2, \frac{dp}{d-sp}))>r $, by the interpolation inequality, we have for a suitable $\alpha=\alpha(\tilde r, r)\in(0,1)$:
\begin{equation*}
\begin{split}
\sup_{[0,T]}\|  \widetilde{v}_\dt - v \|_{L^{\tilde{r}}(\Omega)} & \leq \sup_{[0,T]} \| \widetilde{v}_\dt - v \|_{L^r(\Omega)}^\alpha \| \widetilde{v}_\dt - v  \|_{L^\infty(\Omega)}^{1-\alpha}, \\
& \leq C \sup_{[0,T]} \|\widetilde{v}_\dt - v  \|_{L^r(\Omega)}. 
\end{split}
\end{equation*}
Finally, $\tilde{v}_\dt$ tends to $v$ in $C([0,T],L^r(\Omega))$ for any $r\in [1, +\infty)$.\\
Again using \eqref{11}, we infer that as $\dt$ goes to $0$ : 
\begin{equation}\label{18}
v_\dt,\, v_\dt(\cdot-\dt) \to v \text{ in $L^\infty(0,T,L^r(\Omega))$ for any $r\in [1, +\infty)$.}
\end{equation}
Then, we have by applying dominated convergence theorem: for any $r\in [1,+\infty)$
\begin{equation*}\label{35}
h_R(\cdot,\cdot,v_\dt(\cdot-\dt))=h_R(\cdot,\cdot,v^{n-1}) \to h_R(\cdot,\cdot,v) \text{ in } L^r(Q_T) \mbox{ as }  \dt\to 0.
\end{equation*}
For $r\in [1,+\infty)$, we introduce the linear continuous operator $T_\dt$ from $L^r(Q_T)$ to itself, defined by  
\begin{equation*}
T_\dt f(t,x):=\displaystyle{\fint^{t_n}_{t_{n-1}}}f(\tau,x)d\tau\ \text{ for } (t,x)\in [t_{n-1},t_n)\times \Omega.
\end{equation*} 
 Hence, noting that $h_\dt=T_\dt h_R(\cdot,\cdot,v^{n-1})$ and $T_\dt f$ tends to $f \text{ in } L^{r}(Q_T)$  as $\dt$ goes to $0$, 
 we have 
 \begin{equation*}
\begin{split}
\|h_\dt -&h_R(\cdot,\cdot,v)\|_{L^r(Q_T)}\\
&\leq \|h_\dt -T_\dt h_R(\cdot,\cdot,v)\|_{L^r(Q_T)}+ \|h_R(\cdot,\cdot,v) -T_\dt h_R(\cdot,\cdot,v) \|_{L^r(Q_T)}\\ 
& \leq \|h_R(\cdot,\cdot,v^{n-1}) - h_R(\cdot,\cdot,v)\|_{L^r(Q_T)}+ \|h_R(\cdot,\cdot,v) -T_\dt h_R(\cdot,\cdot,v) \|_{L^r(Q_T)}.
\end{split}
\end{equation*}
Thus, for any $r\geq 1$, $h_\dt$ tends to $h_R(\cdot,\cdot , v)$ in $L^r(Q_T)$.\\
Finally, we prove the convergence  of the term $(\widetilde{\betam(v)}_\dt)_{\dt}$. Since $\beta$ is  globally $\frac1m$-H\"older,  we have firstly 
from \eqref{18},
\begin{equation}\label{22}
\beta(v_\dt)\to \beta (v)\text{ {\it a.e.} in $Q_T$ and in }L^\infty(0,T,L^r(\Omega)) \text{ for any }r\in [1,+\infty)
\end{equation}
and secondly for suitable constant $c>0$:
\begin{equation*}
\max_{n=1,..,N} \frac{1}{\dt} \int_\Omega |\betam(v^{n})-\betam(v^{n-1})|^{m+1}\,dx\leq c\sum^{N}_{n=1} \frac{1}{\dt} \int_\Omega (\betam(v^{n})-\betam(v^{n-1}))(v^{n}-v^{n-1})
\end{equation*}
which implies from \eqref{estiDis}
\begin{equation}\label{21}
\sup_{[0,T]}\|\widetilde{ \betam (v)}_\dt-\betam(v_\dt)\|_{L^{m+1}}^{m+1}\leq C \dt.
\end{equation}
Finally, it follows from \eqref{22} and \eqref{21} that
\begin{equation*}
\widetilde{ \betam (v)}_\dt\to \beta (v)\text{ in }C([0,T],L^r(\Omega)) \text{ for any }r\in [1,+\infty).
\end{equation*}

Passing to the limit as $\dt$ goes to $0$, we obtain that $v\in X_T$ and $v(0)=v_0$ {\it a.e.} in $\Omega$.\\
Since $(v_\dt)$ is bounded in $L^\infty(0,T;\Wspz)$, then

$$V_\dt: (t,x,y) \mapsto \frac{\pow{v_\dt(x)-v_\dt(y)}{p-1}}{|x-y|^{(d+sp)/p'}},$$

is bounded in $L^\infty(0,T;L^{p'}(\mathbb{R}^{2d}))$ uniformly in $\dt$. From the pointwise convergence and weak compactness, we get:

$$V_\dt\overset{*}{\rightharpoonup} \frac{\pow{v(x)-v(y)}{p-1}}{|x-y|^{(d+sp)/p'}} \text{ in } L^\infty(0,T;L^{p'}(\mathbb{R}^{2d})).$$

Finally, for any $\phi \in  L^1(0,T;\Wspz)$, $\frac{\phi(x)-\phi(y)}{|x-y|^{(d+sp)/p}}$ belongs to $ L^1(0,T;L^{p}(\mathbb{R}^{2d})).$ Then, we deduce:
\begin{equation*}\label{24}
\int_0^t \langle\pfrac(v_\dt),\phi\rangle\,d\tau \to \int_0^t \langle\pfrac v,\phi\rangle\,d\tau.
\end{equation*}

\textbf{\underline{Step 5:}} $v$ verifies \eqref{fv2}.\\
By \textbf{Step 2} and \textbf{Step 3}, $\widetilde{\betam(v)}_\dt\in \mathcal W_T\cap L^\infty(Q_T)$, hence applying  Lemma \ref{ipp} for any $\phi \in \mathcal W_T$, we get
$$\int_0^t \int_\Omega \partial_t \widetilde{\betam(v)}_\dt\phi\,dxd\tau=-\int_0^t \int_\Omega  \widetilde{\betam(v)}_\dt\partial_t\phi\,dxd\tau +\bigg[\int_\Omega\widetilde{\betam(v)}_\dt\phi\,dx\bigg]^t_0.$$
Then 
\begin{equation*}
    \begin{split}
\bigg[\int_\Omega\widetilde{\betam(v)}_\dt\phi\,dx\bigg]^t_0-\int_0^t \int_\Omega  \widetilde{\betam(v)}_\dt\partial_t\phi\,dxd\tau+\int_0^t &\langle\pfrac v_\dt,\phi\rangle\,d\tau\\
&=\int_0^t\int_\Omega h_\dt\phi\,dxd\tau,
\end{split}
\end{equation*}
for any $\phi \in\mathcal W_T\cap L^1(0,T;\Wspz)$.\\
Finally, thanks to \textbf{Step 4}, we pass to the limit in the previous equality and we conclude  \eqref{fv2}.
\end{proof}	

\begin{remark}
Assuming $h$ and $v_0$ are nonnegative, Theorem \ref{princcomp} implies that, for any $n$, $v^n$, given by \eqref{pbn}, is nonnegative and thus $v_\dt$ also. Hence, passing to the limit as $\dt$ goes to $0$, we have $v\geq 0$.
\end{remark}

\begin{corollary}\label{coro1}
Let $v_0\in L^\infty(\Omega) \cap \Wspz$ and assume \eqref{f1} for some $q>0$.
Then, there exists $T^\star\in (0,+\infty]$ such that for any $T<T^\star$, \eqref{pbh} admits a weak solution on $Q_T$.\\
Moreover, assume in addition that $q\leq 1$ then there exists $v_{glob}\in L^\infty_{loc}(0, \infty;L^\infty(\Omega))$ such that, for any ${T}<\infty$, $v_{glob}$ is a weak solution of \eqref{pbh} on $Q_{T}$.
\end{corollary}

\begin{proof}
Let $q>0$ and $T>0$. Let $v$ be the solution obtained by Theorem \ref{thexi1} on $Q_{T}$.
By {pointwise} convergence and \eqref{bdd}, we have
$$\|\beta(v)\|_{L^\infty(\Omega)}\leq \|\beta(v_0)\|_{L^\infty(\Omega)}+C_hT(1+R^\frac{q}{m}):=P(R).$$
 We define
\begin{equation}\label{tstar}
T^\star=\sup_{\sigma>0} \frac{\sigma^\frac{1}{m}-\|v_0\|^\frac{1}{m}_{L^\infty(\Omega)}}{C_h(1+\sigma^\frac{q}{m})}\in (0,+\infty].
\end{equation}
By definition of $T^\star$, for any $T<T^\star$, there exists $\tilde R>0$ such that $P(\tilde R)<\tilde R^\frac{1}{m}$ .\\
Hence, with this suitable choice of $\tilde R$, we deduce that $h_{\tilde R}(\cdot,\cdot,v)=h(\cdot,\cdot,v)$ {\it a.e.} in $Q_T$ and $v$ is a weak solution of \eqref{pbh} in $Q_T$.\\
Furthermore, for $q=1$, {$T^*=\frac{1}{C_h}$ and for $q<1$, $T^\star=\infty$}.
Thus, for $q\in (0,1]$ fixed,  we choose $T_o\in (0,\frac{1}{C_h})$ (independent of $v_0$) such that for $R_1$ large enough, $P(R_1)< R_1$. Then, Theorem \ref{thexi1} and arguments above give a weak solution $v_1$ of \eqref{pbh} on {$Q_{T_o}$}.\\
Since $v_1\in X_{T_o}$,  $v_1(T_o)\in L^\infty (\Omega)\cap \Wspz$, thus we apply first Theorem \ref{thexi1}  on $Q_{T_o}$ with the initial data $v_1(T_o)$ and $h(\cdot+T_o,\cdot,\cdot)$  and next we apply the first part of Corollary \ref{coro1} with a suitable $R_2$ large enough to obtain the existence of a weak solution  $v_2$  of $(P_h)$ on $Q_{T_o}$.\\
Hence, by induction argument, we build a sequence of solution $v_n$ of \eqref{pbh} for the data $v_n(0)=v_{n-1}(T)$ and  $h(\cdot+nT,\cdot,\cdot)$. \\
Finally, defining $v_{glob}:= v_{n+1}(\cdot-nT) $ on $[nT,(n+1)T]\times \Omega$, we have $v_{glob}\in L^\infty_{loc}(0, \infty;L^\infty(\Omega))$ such that, for any $\tilde{T}<\infty$, $v_{glob}\in X_{\tilde T}$ is a weak solution of \eqref{pbh} on $Q_{\tilde T}$.
\end{proof}

\subsection{Properties}\label{subsection 3.2}
In this subsection, we establish several properties of weak-mild solutions. {We first consider $h\in L^\infty (Q_T)$ depending only on $x$ and $t$. In this case, the condition \eqref{f1} is not required and applying Theorem \ref{thexi1}, existence of a weak solution to \eqref{pbh} holds for any $T>0$. We begin by proving uniqueness and contraction property:}

\begin{proposition}\label{mildcor}
Let $h\in L^\infty(Q_T)$ and $v$ be the solution of \eqref{pbh} obtained in Theorem \ref{thexi1}, then $v$ is the unique weak-mild solution of \eqref{pbh}. 
\end{proposition} 

\begin{proof}

Let $\epsilon>0$ then, for $\dt$ small enough, $\beta (v_\dt)$ is an $\epsilon$-approximation solution. \\
From Theorem \ref{accretif}, the operator $A:u\mapsto\pfrac \pow{u}{m}$ is accretive in $L^1(\Omega)$ and from Proposition \ref{propdom}, $\beta(v_0)\in \overline{D(A)}^{L^1(\Omega)}$, then using {\cite[Th. 4.1]{barbu2}}, we have the uniqueness of the mild solution of \eqref{pbo} for $h\in L^\infty(Q_T)$ and hence the uniqueness of the weak-mild solution of \eqref{pbh}.
\end{proof}

\begin{remark}\label{rk3}
Due to the uniqueness of the mild solution, we consider in the sequel the weak-mild solution which is the limit of the discretized solutions.
\end{remark}

\begin{proposition}\label{estimations2}
Let $T>0$ and let $u$, $v$ be two weak-mild solutions of \eqref{pbh} with right-hand side respectively $h=f\in {L^\infty(Q_T)}, $  $h=g\in L^\infty(Q_T)$ and initial data $u_0,\,v_0\in { \Wspz\cap L^{\infty}(\Omega)}$. Then, we have 
\begin{equation}\label{L1contr2}
\sup_{[0,T]}\|\beta(u)-\beta(u)\|_{L^1(\Omega)}\leq \|\beta(u_0)-\beta(v_0)\|_{L^1(\Omega)}+\int_0^T\|f-g\|_{L^1(\Omega)}\,dt.
\end{equation}
Moreover, setting $w=u$ or $v$ and $h=f$ or $g$, $w$ is weakly continuous from $[0,T]$ to $\Wspz$ and satisfies
\begin{equation}\label{bd1}
\|\beta(w)\|_{L^\infty(Q_T)}\leq \|\beta(w_0)\|_{L^\infty(\Omega)}+T\|h\|_{L^\infty(Q_T)}
\end{equation} 
and there exists $C>0$ such that:
\begin{equation}\label{energyineq}
\begin{split}
C\int_{t_0}^{t_1} \int_\Omega  \Big|\partial_t\pow{w}{\frac{m+1}{2m}}\Big|^2\,dxdt+&\frac{1}{p}\|w(t_1)\|^p_\Wspz \\
&\leq \frac{1}{p}\|w(t_0)\|^p_\Wspz +\int_{t_0}^{t_1}\int_\Omega h \partial_t w\,dxdt,
\end{split}
\end{equation}
for any $t_0,t_1\in [0,T]$ such that $t_1 \geq t_0$.
\end{proposition}
\begin{proof}
By remark \ref{rk3}, we consider the solutions obtained by Theorem \ref{thexi1}.\\
Let $u_{\dt}$, $v_{\dt}$ be the solutions defined by \eqref{pbn} associated to data $(f,u_0)$ and $(g,v_0)$ respectively. \\
From $L^1$-accretivity of $A$, we get:
\begin{equation*}
\begin{split}
\|\beta(u^{n})-\beta(v^{n})\|_{L^1(\Omega)} & \leq \|\beta(u^{n-1})-\beta(v^{n-1})\|_{L^1(\Omega)}+\dt\|f^n-g^n\|_{L^1(\Omega)},\\
& \leq \|\beta(u_0)-\beta(v_0)\|_{L^1(\Omega)} +\|T_\dt f-T_\dt  g\|_{L^1(Q_T)},\\
& \leq \|\beta(u_0)-\beta(v_0)\|_{L^1(\Omega)} +\| f-g\|_{L^1(Q_T)}.
\end{split}
\end{equation*}
From \eqref{22} and passing to the limit as $\dt$ goes to $0$, we conclude \eqref{L1contr2}.\\
For the end of the proof, we take $w=v$. We now prove \eqref{bd1}. First, inequality \eqref{bdd} stands with $g\in L^\infty(Q_T)$ as
$$\|\beta(v_\dt)\|_{L^\infty(Q_T)} \leq \|\beta(v_0)\|_{L^\infty(\Omega)}+T\|g\|_{L^\infty(Q_T)}.$$
Using again  \eqref{22} and passing to the limit, we obtain \eqref{bd1}.


Finally we prove the energy inequality \eqref{energyineq}. Since the solution  $v$ belongs to $C([0,T];L^r(\Omega))$ for any $r\in [1,+\infty)$ and to $L^\infty(0,T;\Wspz)$, the compact embedding $\Wspz \hookrightarrow  L^p(\Omega)$ entails that $v$ is weakly continuous from $[0,T]$ to $\Wspz$ and 
\begin{equation}\label{14}
\|v(t_0)\|_\Wspz \leq \liminf_{t\to t_0}\|v(t)\|_\Wspz.
\end{equation}
Using \eqref{ineg3} we have:
$$|\beta(v^n)-\beta(v^{n-1})|\geq C\big|\pow{v^n}{\frac{m+1}{2m}}-\pow{v^{n-1}}{\frac{m+1}{2m}}\big|\big(|\pow{v^{n}}{\frac{m+1}{2m}}|+|\pow{v^{n-1}}{\frac{m+1}{2m}}|\big)^{p_1-2}$$
and
$$|v^n-v^{n-1}|\geq C\big|\pow{v^n}{\frac{m+1}{2m}}-\pow{v^{n-1}}{\frac{m+1}{2m}}\big|\big(|\pow{v^{n}}{\frac{m+1}{2m}}|+|\pow{v^{n-1}}{\frac{m+1}{2m}}|\big)^{p_2-2},$$

where $p_1=\frac{2}{m+1}+1$ and $ p_2=\frac{2m}{m+1}+1$. We have $p_1+p_2-4=0$ and so

\begin{equation}\label{13}
(\beta(v^n)-\beta(v^{n-1}))(v^n-v^{n-1})\geq C\big|\pow{v^n}{\frac{m+1}{2m}}-\pow{v^{n-1}}{\frac{m+1}{2m}}\big|^2.
\end{equation}
We now set:
$$ \widetilde{\nu}  _\dt=\dfrac{t-t_{n-1}}{\dt}(\pow{v^n}{\frac{m+1}{2m}}-\pow{v^{n-1}}{\frac{m+1}{2m}})+\pow{v^{n-1}}{\frac{m+1}{2m}} \text{ on } [t_{n-1},t_n[.$$
Following \textbf{Step 3} of the proof of Theorem \ref{thexi1} and using \eqref{13}, we obtain:

\begin{equation}\label{25}
\partial_t \widetilde{\nu}_\dt \rightharpoonup \partial_t\pow{v}{\frac{m+1}{2m}} \text{ in } L^2(Q_T).
\end{equation}

Using again \textbf{Step 3} of the proof of Theorem \ref{thexi1}, \eqref{14}
-\eqref{25} we have for $t\in[0,T]$ such that $t^{N'}$ tends to $t$:

$$C\int_0^t \int_\Omega  \Big|\partial_t\pow{v}{\frac{m+1}{2m}}\Big|^2\,dxd\tau+\frac{1}{p}\|v(t)\|^p_\Wspz \leq \frac{1}{p}\|v_0\|^p_\Wspz +\int_0^t\int_\Omega f \partial_tv\,dxd\tau.$$

Setting $\tilde{v}=v(\cdot+t_0)$ and from the uniqueness of mild solution, we have \eqref{energyineq} for any $t_0\in [0,T)$.
\end{proof}

\begin{proposition}\label{mildth}
Let $T>0$. Let $u$ be the solution of \eqref{pbh} obtained by  Theorem \ref{thexi1}  with right-hand side $h=f\in {L^\infty(Q_T)}\cap  W^{1,1}([0,T];L^1(\Omega)) $ and initial data $u_0\in { \Wspz\cap L^{\infty}(\Omega)}$ verifying $\beta(u_0) \in D(A)$. Then, $\beta (u)\in W^{1,\infty}(0,T;L^1(\Omega))$ and $u$ is a strong solution in sense of Definition \ref{defsol1}.
\end{proposition}

\begin{proof} 
The proof follows the proof of Theorem 4.4 in \cite{barbu2}. For the convenience of the readers, we give it extensively below.\\
The estimate \eqref{L1contr2} and Remark 1.2 of  \cite{barbu2} imply that for any $t_1, t_2\in [0,T]$ such that $t_1<t_2$
\begin{equation*}
    \begin{split}
\|\beta(u)(t_2)&-\beta(u)(t_1)\|_{L^1(\Omega)}\\
&\leq \|\beta(u)(t_2-t_1)- \beta(u_0)\|_{L^1(\Omega)}+\int_0^{t_1}\|f(\tau+t_2-t_1)-f(\tau)\|_{L^1(\Omega)}d\tau,\\
&\leq \|\beta(u)(t_2-t_1)- \beta(u_0)\|_{L^1(\Omega)}+C|t_1-t_2|,
\end{split}
\end{equation*}
and
$$\|\beta(u)(t_2-t_1)-\beta(u_0)\|_{L^1(\Omega)}\leq \int_0^{t_2-t_1} \|f(\tau)-\pfrac u_0\|_{L^1(\Omega)}\,d\tau.$$

From above computations, we deduce that $\beta(v)$ is Lipschitz continuous from $[0,T]$ to $L^1(\Omega)$. Hence, \cite[Theorem 1.17]{barbu2} implies that $\beta (u) \in W^{1,\infty}(0,T,L^1(\Omega))$.\\
Since $L^1(\Omega) \hookrightarrow (L^\infty(\Omega))^*$, we have $L^1(0,T; L^1(\Omega))\hookrightarrow L^1(0,T;(L^\infty(\Omega))^*)$ and with $\beta(v) \in L^\infty(Q_T)$, we get $\beta(v)$ belongs to the Sobolev-Bochner space\break $W^{1,m+1,(m+1)'}(0,T;V,V^*)$ with $V=L^\infty(\Omega)$.\\
Finally  taking  $\phi \in L^{\infty}(Q_T)\cap L^p(0,T;\Wspz)$ with $\partial_t \phi \in L^{(m+1)'}(Q_T)$ in \eqref{fv2}  and  applying Lemma \ref{ipp}  we deduce that  $v$ is a strong solution (in sense of Definition \ref{defsol1}).
\end{proof}

As in \cite{barbu2}, we observe:
\begin{remark}
Let $v_0\in L^\infty(\Omega)\cap\Wspz$ and $v$ be a strong solution then $\beta(v)$ is a mild solution.
\end{remark}

Going back to a more general $h$, we state:
\begin{theorem}[Uniqueness]\label{unipb2}
Let $v_0\in L^\infty(\Omega)\cap \Wspz$ and $T>0$. Assume that \eqref{f1} and \eqref{f3} hold.  Then, \eqref{pbh} admits a unique weak-mild solution on $Q_T$.
\end{theorem}

\begin{proof}
Let $u,v$ be two weak-mild solutions, then $h(\cdot,\cdot,u),h(\cdot,\cdot,v)\in L^\infty(Q_T)$. Then from the uniqueness of the mild solution for a fixed $f\in L^\infty(Q_T)$ and from \eqref{L1contr2} we obtain:
$$\sup_{[0,T]}\|\beta(u)-\beta(v)\|_{L^1(\Omega))}\leq \int_0^T\|h(t,\cdot,u)-h(t,\cdot,v)\|_{L^1(\Omega)}\,dt.$$
Since $u,v\in L^\infty(Q_T)$ and from \eqref{f3}, we infer:
$$\|h(t,\cdot,u)-h(t,\cdot,v)\|_{L^1(\Omega)}\leq C\|\beta(u)-\beta(v)\|_{L^1(\Omega)}.$$
Then, Theorem \ref{unipb2} follows from Lemma \ref{gronw}.
\end{proof}
\subsection{Proofs of Theorems \ref{thgen} and \ref{GB2}}
Taking into account the results in subsections \ref{subsection 3.1} and \ref{subsection 3.2}, we show Theorems \ref{thgen} and \ref{GB2}.

\begin{proof}[Proof of Theorem \ref{thgen}]
First, Corollary \ref{coro1} gives the existence of a weak solution $v$ of \eqref{pbh} on $Q_T$ for some suitable $T>0$. Using Proposition \ref{mildcor} and considering $f=h(\cdot,\cdot,v)\in L^\infty(Q_T)$, we deduce that $\beta(v)$ is a mild solution of \eqref{pborig} and hence $v$ is a $T$-weak-mild solution of \eqref{pbh}.\\
\ \\
Proof of {\it (i)}. Let $q\leq 1$ and $v_{glob}\in L_{loc}^\infty(0,\infty;L^\infty(\Omega))$ be the solution given in Corollary \ref{coro1}.\\ 
As above, Proposition \ref{mildcor} implies that $v_{glob}$ is a mild solution on each interval $[nT,(n+1)T]$ for $n\in \mathbb N$ where the parameter $T$ is given as in the proof of Corollary \ref{coro1}. \\
Hence, consider for instance the intervals $[0,T]$ and $[T,2T]$. From the definition of an $\epsilon$-approximation and since $\tilde v _\dt$ tends to $v_{glob} $ in $C([0,T];L^r(\Omega))$, we have that $v_{glob}$ is a mild solution over $[0,2T]$. By induction argument, we obtain that, for any $\hat T>0$, $v_{glob}$ is a weak-mild solution of \eqref{pbh} on $Q_{\hat T}$.
Finally, we deduce the existence of a $\infty$-weak-mild solution.\\
\ \\
Proof of {\it (ii)}. Let $T'\in (0,T)$, {\it(ii)} follows from the inequality \eqref{L1contr2} with $f=h(\cdot,\cdot,v)$ and $g=h(\cdot,\cdot,u)$ belonging to $L^\infty(Q_{T'})$. \\
\ \\
Proof of {\it (iii)}.  Under the condition \eqref{f3}, Theorem \ref{unipb2} infers the uniqueness of the weak-mild solution that entails the uniqueness of the $T$-weak-mild solution.\\
\ \\
Proof of {\it (iv)}. Assume that $\beta(v_0)\in D(A)$ and $\tilde T\in (0,T)$ providing  $h(\cdot,\cdot,v)\in W^{1,\infty}(0,\tilde T;L^1(\Omega))$. Then, Proposition \ref{mildth} with $f=h(\cdot,\cdot,v)$  implies that $v$ is a strong solution of \eqref{pbh}.\
\end{proof}

\begin{proof}[Proof of Theorem \ref{GB2}] Let $T\in (0,+\infty)$. Without loss of generality, we can assume $T\geq T^\star$ where $T^\star$ is given in \eqref{tstar}.\\
Let $v$ be the unique $T$-weak-mild solution of \eqref{pbh}. We assume $||v(t)||_{L^\infty(\Omega)} \nrightarrow \infty$ as $t$ goes to $T$ {\it i.e.} there exits a constant $\tilde C>0$ such that 
\begin{equation}\label{linf}
\liminf_{t\to T}||v(t)||_{L^\infty(\Omega)}\leq \tilde C.
\end{equation}
%
%
Define  $C_0:=||v_0||_{L^\infty(\Omega)}^\frac{1}{m}+\tilde{C}^\frac{1}{m}<\infty$ and $T_0:= \sup_{\sigma>0} \frac{\sigma^\frac{1}{m}-C_0}{C_h(1+\sigma^\frac{q}{m})}<T^\star.$\\
Since \eqref{linf} holds, there exists $t_0\in (T-\frac12 T_0,T)$ such that $||v(t_0)||_{L^\infty(\Omega)}\leq \tilde C$. Hence, defining  
$$\tilde T=\sup_{\sigma>0} \frac{\sigma^\frac{1}{m}-||v(t_0)||^\frac{1}{m}_{L^\infty(\Omega)}}{C_h(1+\sigma^\frac{q}{m})}> T_0,$$
Corollary \ref{coro1} gives the existence of a unique weak-mild-solution $\tilde{v}$ of \eqref{pbh} for $\tilde{v}(0)=v(t_0)$ on $[0,\tilde T-\epsilon]\times \Omega$  for any $\epsilon>0$ small enough.\\
Finally we define  $V=v$ on $[0,t_0]$ and $V=\tilde{v}$ on $[t_0, t_0+\tilde T-\epsilon]$. Hence choosing $\epsilon$ small enough such that $T'=t_0+\tilde T-\epsilon>T$ and thanks to Theorem \ref{unipb2}, we deduce that  $V$ extends the solution $v$ of \eqref{pbh} on $[0,T']\times \Omega$ with $T'>T$.
\end{proof}

\subsection{Sub-homogeneous case}\label{subh}

In this subsection,  we always assume $p>\frac{q}{m}+1$ in the condition \eqref{f1}.
\begin{proof}[Proof of Theorem \ref{subhomo2}] By Theorem \ref{thexi1}, for $R>0$ and $T>0$, there exists a {nonnegative} weak solution $v$ of \eqref{pbpsi}.
Next, we show that the hypothesis $0\leq v_0\leq C d(\cdot,\partial \Omega)^s$ implies that the solution $v$ is uniformly bounded in $Q_T$ independently of $T$.\\
Indeed, let $K>0$ and consider $w_K\in C^s(\Rn)\cap \Wspz$ (see \cite[Th 2.7]{reghold} for H\"older-regularity of $w_K$) the unique solution to:
\begin{equation*} 
    \begin{cases}
 \pfrac w_K=K & \text{in} \;\Omega,\\
      w_K=0 & \text{in} \; \Rn\backslash \Omega.
    \end{cases}\
\end{equation*}
From the homogeneity of  $(-\Delta)^s_p$ and from \cite[Theorem 1.5]{delpezzo-quaas}, we have  that, for $K$ large enough, $w_K \geq c_Kd(\cdot,\partial \Omega)^s\geq v_0$ (see the proof of \cite[Theorem 3.1]{exi} for further details).\\
Fixing such $K$, we now consider  $w\in L^\infty(\Omega)\cap \Wspz$  {the positive} solution of:
\begin{equation*} 
    \begin{cases}
 \pfrac w=K +C_h(1+|w|^\frac{q}{m}) & \text{in} \;\Omega,\\
      w=0 & \text{in} \; \Rn\backslash \Omega.
    \end{cases}\
\end{equation*}
Since $p>\frac{q}{m}+1$, the solution $w$ {is unique (from similar arguments as in the proof of Theorem \ref{existab})} and can be obtained as the global minimizer of the coercive energy functional defined on $W^{s,p}_0(\Omega)$ by:
$$J(v)=\frac{1}{p}\|v\|^p_\Wspz-\int_\Omega Kv+C_h(v+\frac{m}{q+m}\pow{v}{\frac{q}{m}+1})\,dx.$$
By Proposition \ref{princcomp} (with $g\equiv 0$),  we obtain that $w\geq w_K \geq v_0$. Moreover using $p>\frac{q}{m}+1$ and the same arguments as \cite[Th.2.2]{exi},  we have $w\in L^\infty(\Omega)$.\\
From $w\geq v_0$, we claim that  $w\geq v^1$. Indeed, we have with \eqref{f1}:
\begin{equation*}
\begin{split}
\frac{\beta(v^1)-\beta(w)}{\dt}&+\pfrac v^1-\pfrac w  \\
&=\frac{\beta(v_0)-\beta(w)}{\dt}+h^0-(K +C_h(1+|w|^\frac{q}{m})) \leq 0
\end{split}
\end{equation*}
and by Proposition \ref{princcomp} with $g=\beta$, we deduce $v^1\leq w$ and by induction $v^n\leq w$ for any $n$.\\
So $(v_\dt)_\dt$ is uniformly bounded in $L^\infty(Q_T)$ by $\|w\|_{L^\infty(\Omega)}$.\\
Finally we choose $R\geq\|w\|_{L^\infty(\Omega)}$ to obtain a solution of \eqref{pbh} on $Q_T$.
\ \\
We get a weak-mild solution for any $T>0$ and as in Corollary \ref{coro1} we obtain a $\infty$-weak-mild solution $v\in L^\infty(Q_\infty)$. In particular we fix $T=1$ and we define $v_{glob}$ on $\Omega\times [0,\infty)$ by $v_{glob}=v_{n+1}(\cdot-n)$ over $[n,n+1]$.
\end{proof}

We have the following comparison principle for the solutions built in the proof of Theorem \ref{subhomo2}:

\begin{theorem}\label{comp}[Comparison Principle]
Let $h$ be nondecreasing and $u,v$ be two solutions of Theorem \ref{subhomo2} associated to the initial data: $u_0,v_0\in L^\infty(\Omega)\cap \Wspz$ nonnegative, respectively. Then if $v_0\geq u_0$, $v\geq u$ on $Q_T$ for any $T>0$.

\end{theorem}

\begin{proof}
We denote $v_m, u_m$ the sequences build at the end of the proof of Theorem \ref{subhomo2} such that $v,u=v_{m+1}(\cdot-m),u_{m+1}(\cdot-m)$ over $[m,m+1]$ .

Using the time discretization scheme \eqref{pbn} and $h$ nondecreasing we have by Proposition \ref{princcomp} that if $v_m(0)\geq u_m(0)$ then $v_m^1 \geq u_m^1$. And by iteration we have $v_m^n\geq u_m^n$ for any $n$, passing to the limit as $\dt$ goes to $0^+$ in $C([0,1],L^r(\Omega))$, we have $v_m\geq u_m$ which gives $v\geq u$.
\end{proof}

\section{The problem \texorpdfstring{\eqref{pbq}}{Pq}}\label{mainpart}

In this section, we study a particular case of the equation \eqref{pbh} considering $h(\cdot,\cdot,\theta)=\pow{\theta}{\frac{q}{m}}$. More precisely, we investigate the following problem:
\begin{equation}\tag{$P_q$}
    \begin{cases}
     \partial_t \beta(v)+\pfrac v=\pow{v}{\frac{q}{m}} & \text{in} \;Q_T,\\
      v=0 & \text{in} \; (0,T) \times \Rn\backslash \Omega, \\
      v(0,\cdot)=v_0& \text{in} \; \Omega ,
    \end{cases}\
\end{equation}

for $q>0$, $m>1$ and $v_0\in \Wspz \cap L^\infty(\Omega)$.\\
From Theorem \ref{thgen}, there exists $T\in (0,\infty]$ such that \eqref{pbq} admits a $T$-weak-mild solution. If $q\in (0,1]$, then $T=\infty$ and if $q\geq 1$ the solution is unique. \\ 
We first establish the following {pointwise type energy estimate}:

\begin{proposition}\label{estim}
Let $T>0$  and $v$ be a weak-mild solution of \eqref{pbq} on $Q_T$. Then, for any $r\geq 1$, the following {energy} estimate holds:
\begin{equation}\label{condex}
\frac{1}{1+mr}\int_\Omega \partial_t (|v|^{r+\frac{1}{m}})\,dx +\langle\pfrac v,\pow{v}{r}\rangle=\int_\Omega |v|^{\frac{q}{m}+r}\,dx, 
\end{equation}
 for {\it a.e.} $t< T$.
\end{proposition}

\begin{proof}
Let $r\geq1$. By Definition \ref{defsol2}, $v\in X_T$ and so $\pow{v}{r}$ is a suitable test function in \eqref{fv2}, then:

\begin{equation*}
\begin{split}
\bigg[\int_\Omega |v|^{r+\frac{1}{m}}\,dx \bigg]^t_0-\int_0^t\int_\Omega \beta(v)\partial_t \pow{v}{r}\,dxd\tau +\int_0^t&\langle\pfrac v,\pow{v}{r}\rangle\,d\tau\\
&=\int_0^t\int_{\Omega} |v|^{\frac{q}{m}+r}\,dxd\tau.
\end{split}
\end{equation*}

Since $\partial_t v \in L^2(Q_T)$ and $v\in L^\infty(Q_T)$, we have $\beta(v) \partial_t \pow{v}{r}=\frac{mr}{1+mr}\partial_t |v|^{r+\frac{1}{m}}$ and applying Lemma \ref{ipp}, this yields 
\begin{equation*}
\begin{split}
\bigg[\int_\Omega |v|^{r+\frac{1}{m}}\,dx \bigg]^t_0-\int_0^t\int_\Omega \beta(v) \partial_t \pow{v}{r}\,dxd\tau &=\frac{1}{1+mr} \bigg[\int_\Omega |v|^{r+\frac{1}{m}}\,dx \bigg]^t_0,\\
&=\frac{1}{1+mr}\int_0^t\int_\Omega \partial_t |v|^{r+\frac{1}{m}}\,dxd\tau.
\end{split}
\end{equation*}
Replacing $t$ by $t+h$, subtracting both expressions and using \cite[Proposition 1.4.29]{c&h}, we derive \eqref{condex} by passing to the limit as $h$ goes to $0$.
\end{proof}

\subsection{Sub-homogeneous case}\label{partsubho}

In this subsection,  we always assume $p>\frac{q}{m}+1$ and we prove Theorem \ref{subhomo1}. Using results in Section \ref{subh} we directly have:

\begin{corollary}\label{subhomo}
Let $v_0 \in L^{\infty}(\Omega)\cap \Wspz$ satisfying $ |v_0|\leq C d(\cdot,\partial \Omega)^s$ for some constant $C>0$. Then, there exists  a $\infty$-weak-mild solution  $v_{glob}\in L^\infty(Q_\infty)$ to \eqref{pbq}.\\
Furthermore, if $u,v$ are solutions respectively associated to initial data: $u_0,v_0\in L^\infty(\Omega)\cap \Wspz$, then $v_0\geq u_0$ implies $v\geq u$ on $Q_T$ for any $T>0$.
\end{corollary}

In the following parts, we take the initial data $v_0$ nonnegative hence the considered solution of \eqref{pbq} are nonnegative. So we show the convergence of the solution of Corollary \ref{subhomo} to a positive steady state solution.

\subsubsection{Existence of a nontrivial stationary solution}

In this paragraph, we show the existence of a nontrivial stationary solution of \eqref{pbq}. Namely, we investigate the positive solution of the following problem: 
\begin{equation}\tag{$Q_{stat}$}\label{15}
    \begin{cases}
 \pfrac v_{\infty}=|v_{\infty}|^\frac{q}{m} & \text{in} \;\Omega,\\
      v_{\infty}=0 & \text{in} \; \Rn\backslash \Omega.
    \end{cases}\
\end{equation}
First, we recall the following notions:
\begin{definition}
We call a supersolution of \eqref{15} (respectively subsolution) any function  $v\in \Wspz$ verifying for any $ \phi \in \Wspz$ and nonnegative:

$$\langle\pfrac v,\phi  \rangle \geq \int_\Omega|v|^\frac{q}{m} \phi\, dx,$$

(respectively $\langle\pfrac v,\phi  \rangle \leq \int_\Omega|v|^\frac{q}{m} \phi \,dx$).
We call (weak) solution to \eqref{15} a function which is both a sub- and supersolution.
\end{definition}

We now prove the following existence result:

\begin{theorem}\label{existab}There exists a unique nontrivial solution  $v_{\infty}\in L^\infty(\Omega)\cap \Wspz$  of \eqref{15}. 
Furthermore, $v_{\infty}$ belongs to $C^s(\mathbb{R}^d)$ and there exists $c>0$ such that $\frac1cd(x, \partial\Omega)^s\leq v_{\infty}(x)\leq cd(x,\partial \Omega)^s$.
\end{theorem}

\begin{proof}
For the existence of a solution, we minimize the associated energy functional $J$ defined for any $v\in\Wspz$ by:
$$J(v)=\frac{1}{p}\|v\|^p_\Wspz-\int_\Omega \frac{m}{q+m}\pow{v}{\frac{q}{m}+1}.$$
Since $p>\frac{q}{m}+1$, $J$ is well-defined,  coercive and {\it w.l.s.c.}. Therefore, there exists  a global minimizer $v_{\infty}$ of $J$ on $\Wspz$. $J$ is G\^ateaux differentiable on $\Wspz$ and hence $v_{\infty}$ is a critical point of $J$ and a solution of \eqref{15}.\\
Furthermore, using Proposition \ref{princcomp} with $g\equiv 0$, we obtain that $v_{\infty}$ is nonnegative. We also have $v_{\infty}\neq 0$. Indeed, let  $v\in\Wspz$, $v> 0$ and $t>0$ small enough, we have $J(v_{\infty})\leq J(tv)<0=J(0)$.\\
Using the same arguments as \cite[Th. 2.2]{exi}, we have $v_{\infty}\in L^\infty(\Omega)$ and by \cite[Th. 2.7]{reghold} we have $v_{\infty}\in C^s(\mathbb{R}^d)$ and then $v_{\infty}\leq Cd(\cdot, \partial\Omega)^s$ for some $C>0$. By the strong maximum principle (see \cite{delpezzo-quaas}), we obtain that $v_{\infty}>0$ in $\Omega$. \\
Then, \cite[Theorem 1.5]{delpezzo-quaas} gives $v_{\infty}\geq  cd(\cdot, \partial\Omega)^s$ for some $c>0$.\\
Finally, we show the uniqueness of the weak solution to \eqref{15}. \\
Let $u$, $v$ be two solutions to \eqref{15}. Then $u$, $v\in [cd(\cdot, \partial\Omega)^s,Cd(\cdot, \partial\Omega)^s]$  and  $u-\frac{v^p}{u^{p-1 }}$ and $v-\frac{u^p}{v^{p-1 }}$ belong to $L^\infty (\Omega)\cap \Wspz$. Hence
\begin{equation*}\label{16}
\langle\pfrac u, u-\frac{v^p}{u^{p-1 }}\rangle=\int_\Omega u^\frac{q}{m}(u-\frac{v^p}{u^{p-1 }})\,dx
\end{equation*}
and
\begin{equation*}\label{17}
\langle\pfrac v, v-\frac{u^p}{v^{p-1 }}\rangle=\int_\Omega v^\frac{q}{m}(v-\frac{u^p}{v^{p-1 }})\,dx.
\end{equation*}
Summing both previous equalities and since  $p>\frac{q}{m}+1$, we get:
\begin{equation*}\label{34}
\begin{split}
\langle\pfrac u, u-\frac{v^p}{u^{p-1 }}\rangle+\langle&\pfrac v, v-\frac{u^p}{v^{p-1 }}\rangle \\
&=\int_\Omega (u^p-v^p)(\frac{u^\frac{q}{m}}{u^{p-1}}-\frac{v^\frac{q}{m}}{v^{p-1}})\,dx\leq 0.
\end{split}
\end{equation*}
\cite[Lemma 1.8]{exi} and the above inequality imply that there exists $k$ such that $u=kv$ and since $p\neq \frac{q}{m}+1$, we have necessary $k=1$ and the proof follows.
\end{proof}
\subsubsection{Sub-supersolution method}


\begin{proof}[Proof of Theorem \ref{subhomo1}]

The existence of a $\infty$-weak-mild solution is given by Theorem \ref{subhomo2} (or Corollary \ref{subhomo}). It remains to obtain the the asymptotic behaviour.\\
For this, using monotone arguments, we proceed in several steps.\\
\underline{\textbf{Step 1:}} We first find $\underline{v}_0$ and $\overline{v}_0$ respectively subsolution and supersolution of \eqref{15} such that $v_0\in [\underline{v}_{0},\overline{v}_{0}]$.\\
As in the proof of Theorem \ref{subhomo2}, there exists a unique $\overline{v}_0 \in L^\infty(\Omega)\cap \Wspz$ solution of
\begin{equation*} 
    \begin{cases}
 \pfrac \overline{v}_0=K +|\overline{v}_0|^\frac{q}{m} & \text{in} \;\Omega,\\
      \overline{v}_0=0 & \text{in} \; \Rn\backslash \Omega,
    \end{cases}\
\end{equation*}
with $K$ large enough such that $v_0\leq \overline{v}_0\leq \kappa d(\cdot,\partial \Omega)^s$ for some $\kappa>0$.\\
We now define $\underline{v}_0=\lambda v_{\infty}$, where $v_{\infty}$ is the nontrivial nonnegative  solution of \eqref{15} 
then for $\lambda \leq 1$ small enough $\underline{v}_0\leq v_0$ and since $p>\frac{q}{m}+1$, $\underline{v}_0$ is a subsolution of \eqref{15}.\\
\underline{\textbf{Step 2:}} For any $T>0$, \eqref{pbq} admits a weak-mild solution $\overline v$ (resp. $\underline v$) on $Q_T$ such that $\overline v(0)=\overline v_0$ (resp. $\underline v(0)=\underline v_0$). Moreover $\overline v$ is nonincreasing (resp. $\underline v$ nondecreasing) in time and 
$\overline v(T)$ is a supersolution (resp. $\underline v(T)$ is a subsolution)  of \eqref{15}.\\
The proof are similar for both cases and we only consider the proof for $\overline v$.\\
The existence of $\overline v$ is given by Theorem \ref{subhomo2}. \\
Let $(\overline{v}^n)_n$ be the sequence defined as in Step 1 of the proof of Theorem \ref{thexi1} with $\overline v^0=\overline v_0$. Since $\overline{v}_0$ is a supersolution of \eqref{15}, we deduce from the definition of $\overline{v}^1$ that for any nonnegative $\phi\in \Wspz$:
$$ \int_\Omega \beta(\overline{v}^1)\phi\,dx+\dt\langle\pfrac\overline{v}^1,\phi\rangle \leq \int_\Omega \beta(\overline{v}_0)\phi\,dx+\dt\langle\pfrac\overline{v}_0,\phi\rangle.$$
Therefore,  by Proposition \ref{princcomp}, we get $0\leq \overline{v}^1\leq \overline{v}_0$ {\it a.e.} in $\Omega$. \\
By induction, we obtain, for any $n$, $0\leq\overline{v}^{n+1}\leq \overline{v}^n$ and hence we deduce that $\overline{v}_\dt$ and then $v$ are nonincreasing in time.\\
Furthermore, for any $n$, we have that $\overline{v}^{n}$ is a supersolution of \eqref{15}, that is, for any nonnegative $\phi\in \Wspz$:
\begin{equation*}
\begin{split}
\langle \pfrac \overline{v}^{n+1},\phi\rangle & \geq  \int_\Omega\frac{\beta(\overline{v}^{n+1})-\beta(\overline{v}^{n})}{\dt}\phi\,dx+\langle \pfrac \overline{v}^{n+1},\phi\rangle \\
&= \int_\Omega (\overline{v}^n)^\frac{q}{m}\phi\,dx \geq \int_\Omega (\overline{v}^{n+1})^\frac{q}{m}\phi\,dx.
\end{split}
\end{equation*}
Note that since $\overline{v}_\dt$ is uniformly bounded in $L^\infty(Q_T)$, we can get the convergence results similarly as in the proof of Theorem \ref{thexi1}. In particular $\overline{v}_\dt$ tends to $\overline{v} $ in $L^\infty(0,T;L^r(\Omega))$. Also using $\overline{v}^N$ bounded in $\Wspz$ and $\overline{v}_\dt(T)$ goes to $ \overline{v}(T) $ in $L^r(\Omega)$ we have, as in Step 4 of the proof of Theorem \ref{thexi1}, $\langle \pfrac \overline{v}^N, \phi \rangle$ tends to $\langle \pfrac \overline{v}(T) , \phi \rangle$ for any $\phi \in  \Wspz.$ Next, since 
$$ \langle\pfrac \overline{v}^N, \phi \rangle \geq \int_\Omega |\overline{v}^N|^\frac{q}{m} \phi, \; \forall \phi\in \Wspz \text{ nonnegative }$$
and passing to the limit as $N$ goes to $\infty$, we get that $\overline{v}(T)$ is a supersolution of \eqref{15}.
\\
\underline{\textbf{Step 3}}
Let $v_{glob}\in L^\infty(Q_\infty)$ be the solution of Theorem \ref{subhomo2} with initial data $ v_0$.\\
Using \textbf{Step 2} and the same construction as in the end of the proof of Theorem \ref{subhomo2}, we obtain $\underline{v}_{glob}$ and $\overline{v}_{glob}$ both $\infty$-weak-mild solutions of \eqref{pbq} with initial data $\underline v_0$ and $\overline v_0$, respectively.

Then, from Theorem \ref{comp}, we infer $v_{glob}\in [\underline{v}_{glob},\overline{v}_{glob}]$ and \textbf{Step 2} implies by an induction argument that $\underline{v}_{glob}$ and $\overline{v}_{glob}$ are respectively nondecreasing and nonincreasing in time and belong to $C([0,\infty],L^r(\Omega))$ for any $r\geq 1$.\\
Thus $\underline{v}_{glob},\overline{v}_{glob}$ tends to $\underline{v}_\infty,\overline{v}_\infty$ {\it a.e.} and in $L^r(\Omega)$ for any $r$ by dominated convergence theorem.\\
We now use semigroup theory arguments to show $\underline{v}_\infty=\overline{v}_\infty=v_{\infty}$.\\
We introduce the family $\{S(t);t\geq 0\}$ on $\{v\in \Wspz\cap L^\infty(\Omega)\ |\ 0\leq v\leq Cd(\cdot,\Omega)^s \mbox{ for some $C>0$}\}$ defined as $S(t)v_0=v_{glob}(t)$ where $v_{glob}$ is the solution of Theorem \ref{subhomo2}. From the construction of $v_{glob}$ at the end of the proof of Theorem \ref{subhomo2}, we have for $t\geq 0$ and $n\in \mathbb{N}$: $S(t+n)v_0=S(t) (S(n)v_0 )$.\\
Since $\overline{v}_{glob}\in C([0,\infty],L^1(\Omega))$,  we get
$$\overline{v}_\infty=\lim_{n\to \infty} S(t+n)v_0=S(t)(\lim_{n\to \infty} S(n)v_0)=S(t)\overline{v}_\infty$$
which means that $\overline v_\infty$ is a nonnegative stationary solution of \eqref{pbq} and by Theorem \ref{existab}, $\overline{v}_\infty=v_\infty$. Similarly, we obtain $\underline{v}_\infty=v_{\infty}$.
Since $v_{glob}\in [\underline{v}_{glob},\overline{v}_{glob}]$, we have $v$ tends to $v_{\infty}$ in $L^r(\Omega)$ for all $r$ as $t$ goes to $\infty$.
\end{proof}

\subsection{Asymptotic behaviour, energy method}\label{partGB}

In this section, we prove asymptotic behaviour in the case $p<\frac{q}{m}+1,$ we use some classical energy methods as {\it e.g.} in \cite{critical}.

\subsubsection{Extinction}

\begin{proof}[Proof of Theorem \ref{GB1}]
Let $v$ be a $\infty$-weak-mild solution of \eqref{pbq} then for any $T>0$, $v$ is a weak-mild solution of \eqref{pbq} on $Q_T$. The following computations are done for any $T>0$ for {\it a.e.} $t\in [0,T)$ which implies  that they hold for {\it a.e.} $t>0$.\\
Let $\rho\geq 1$, we define $\tilde{p}=p-1+\rho$, $\tilde{s}=\dfrac{sp}{\tilde{p}}$ and  




$$r=\min \{ \rho \geq 1 \; | \;  W^{\tilde{s},\tilde{p}}_0(\Omega) \hookrightarrow L^{\frac{1}{m}+\rho}(\Omega)\}.$$


From \ref{inegalg}, and since $(v(x)-v(y))(\pow{v(x)}{r}-\pow{v(y)}{r})\geq c|v(x)-v(y)|^{r+1}$ where $c$ depends on $p,s,d$, we get:
\begin{equation}\label{20}
\begin{split}
\langle \pfrac v,\pow{v}{r} \rangle
& \geq c \int_\Rn \int_\Rn \frac{|v(x)-v(y)|^{p-1+r}}{|x-y|^{d+sp}}\,dxdy=c\|v\|^{\tilde{p}}_{W^{\tilde{s},\tilde{p}}_0}.
\end{split}
\end{equation}
With the definition $r$, one has $W^{\tilde{s},\tilde{p}}_0(\Omega)\hookrightarrow L^{r+\frac{1}{m}}(\Omega)$. Hence, using \eqref{20}, equality \eqref{condex} yields for suitable constants and for {\it a.e } $t>0$:
$$\partial_t \|v\|^{r+\frac{1}{m}}_{L^{r+\frac{1}{m}}(\Omega)}+c\|v\|_{L^{r+\frac{1}{m}}(\Omega)}^{p-1+r}\leq \tilde{c}\|v\|^{r+\frac{q}{m}}_{L^{r+\frac{q}{m}}(\Omega)}\leq C\|v\|^{r+\frac{q}{m}}_{L^{r+\frac{1}{m}}(\Omega)}.$$
Defining the continuous function $M_r\,:\, [0,\infty)\ni t\mapsto\|v(t)\|_{{L^{r+\frac{1}{m}}(\Omega)}}^{r+\frac{1}{m}}$, we get equivalently for {\it a.e} $t>0$:
\begin{equation}\label{19}
M_r'(t) \leq M_r'(t)+\frac{c}{2} M^\alpha_r(t) \leq CM^\gamma_r(t)-\frac{c}{2} M^\alpha_r(t):=Q(M_r(t)) 
\end{equation}
where $\alpha :=\frac{m(p-1+r)}{rm+1}<\frac{rm+q}{rm+1}:=\gamma<1$ since $p<\frac{q}{m}+1$ and $q\leq 1$.\\
There exists $Q_0>0$ such that the mapping $Q$ is negative on $(0,Q_0)$. Hence choosing $M_r(0)=\|v_0\|_{L^{r+\frac{1}{m}}(\Omega)}<Q_0$ and  by continuity of $M_r$, we deduce from \eqref{19} that $M_r$ is nonincreasing on $[0,\varepsilon]$ for $\varepsilon$ small enough and $M_r(\varepsilon)<Q_0$. Hence, reiterating the argument, we obtain that $M_r$ is decreasing on $[0,+\infty)$ hence  for any $t>0$, we have $M_r(t)<Q_0$  and by \eqref{19}
$$M_r'(t)+\frac{c}{2} M_r(t)^\alpha \leq 0. $$
Assume that $M_r(t)>0$ for any $t>0$, then we get
$$\frac{M_r'(t)}{M^\alpha_r(t)}+\frac{c}{2} \leq 0 $$
from which we infer that $M_r(t)\leq (M^{1-\alpha}_r(0)-\dfrac{c}{2}(1-\alpha)t)^\frac{1}{1-\alpha}$ goes to $ -\infty$ as $t$ goes to $+\infty$. Thus we get a contradiction. Then, there exists $t_0>0$ such that $M_r(t_0)=0$ and using the monotony of $M_r$, we are done.
\end{proof}
\subsubsection{Blow-up}
Before going into the entire proof of Theorem \ref{GB3}, we first exhibit an example of function $v$ such that $E(v)\leq 0$.\\
Let $v$ be a nonnegative subsolution of \eqref{15} thus taking $v$ as test function, we have  $\| v\|^p_\Wspz \leq \| v\|^{\frac{q}{m}+1}_{L^{\frac{q}{m}+1}(\Omega)}$. Hence, using $p<\frac{q}{m}+1$ we get $E(\lambda v)\leq 0$ for $\lambda$ large enough.

\begin{proof}[Proof of Theorem \ref{GB3}]
Let $v$ is a weak-mild solution of \eqref{pbq} on $Q_T$ such that $v(0)=v_0\in L^\infty(\Omega)\cap\Wspz$. We first show that if $E(v_0)\leq 0$ then for any $t\in(0,T),$ $ E(v(t))\leq 0$.\\
Fixing $0<t<T$ and using \eqref{energyineq} with $h= \pow{v}{\frac{q}{m}}$, we obtain:
\begin{equation}\label{31}
\begin{split}
C\int_{0 }^{t}\int_\Omega \Big|\partial_t\pow{v}{\frac{m+1}{2m}}\Big|^2 \,dxd\tau+&\frac{1}{p}\|v(t)\|^p_\Wspz\\
&\leq \frac1p \|v_0\|^p_\Wspz+\int_{0 }^{t}\int_\Omega \pow{v}{\frac{q}{m}} \partial_t v\,dxd\tau.
   \end{split}
\end{equation}

As in Proposition \ref{estim}, we establish
$$\int_{0 }^{t}\int_\Omega \pow{v}{\frac{q}{m}} \partial_t v\,dxd\tau=\frac{m}{m+q}\left(\|v(t)\|^{\frac{q}{m}+1}_{L^{{\frac{q}{m}+1}}(\Omega)}-\|v_0\|^{\frac{q}{m}+1}_{L^{{\frac{q}{m}+1}}(\Omega)}\right).$$

Therefore, \eqref{31} gives:

\begin{equation*}
C\int_{0 }^{t}\int_\Omega \Big|\partial_t\pow{v}{\frac{m+1}{2m}}\Big|^2\,dxd\tau +E(v(t))\leq E(v_0)
\end{equation*}
and then $E(v(t))\leq E(v_0) \leq 0$.\\
Now for any $t\in[0,T)$, we define $Y(t):=\frac{1}{m+1}\|v(t)\|^{\frac{1}{m}+1}_{L^{{\frac{1}{m}+1}}(\Omega)}$. From \eqref{condex} with $r=1$, we deduce that:
\begin{equation*}
Y'(t)= \frac{1}{m+1}\partial_t\int_\Omega |v|^{\frac{1}{m}+1}\,dx=-\|v\|^p_\Wspz+\|v\|^{\frac{q}{m}+1}_{L^{\frac{q}{m}+1}(\Omega)},
\end{equation*}
and since $E(v(t))\leq 0$, we have 

\begin{equation}\label{30}
\begin{split}
Y'(t) 
&=-pE(v(t))+(1-\frac{pm}{q+m})\|v\|^{\frac{q}{m}+1}_{L^{\frac{q}{m}+1}(\Omega)}\geq C\|v\|^{\frac{q}{m}+1}_{L^{\frac{q}{m}+1}(\Omega)}\\
\end{split}
\end{equation}
where $C=1-\frac{pm}{q+m}>0$ since $p<\frac{q}{m}+1$.\\ 

\textbf{\underline{Case 1: $q> 1$:}}

In this case, we recall that $v$ is the unique $T_{max}$-weak-mild of \eqref{pbq}.\\
Assume that $T_{max}=+\infty$ hence \eqref{30} holds for any $t>0$ and by H\"older inequality, we get
%
$$Y'(t)\geq \tilde C Y(t)^\frac{m+q}{m+1}.$$
Setting $\nu =\frac{m+q}{m+1}>1$, we deduce
%
%
by straightforward integration that \break $Y^{1-\nu}(t)\leq  -\tilde C(\nu-1) t +Y^{1-\nu}(0)$.\\
Hence for $t$ large enough, we have $Y(t)<0$ which gives a contradiction.\\
Thus, $T_{max}<\infty$ and  Theorem \ref{GB2} implies 
$$\lim_{t\to T_{max}} \|v(t)\|_{L^{\infty}(\Omega)}=+\infty.$$

For the case $q\leq 1$, {according to assertion (i) of Theorem \ref{thgen}, there exists $v$  a $\infty$-weak-mild solution of \eqref{pbq} with initial data $v_0$}. As above, inequality \eqref{30} holds for any $t>0$.\\
\textbf{\underline{Case 2: $q= 1$:}}

Inequality \eqref{30} becomes $Y'(t) \geq CY(t)$ and by integrating we deduce that $Y(t)$ goes to $ \infty $ as $t$ goes to $\infty$. 
\\

\textbf{\underline{Case 3: $q< 1$:}}

Fixing arbitrary $T>0$, $v$ is a weak-mild solution on $Q_T$ and satisfies for any $t\in [0,T]$:

$$\|v(t)\|^{\frac{1}{m}+1}_{L^{\frac{1}{m}+1}(\Omega)}\leq \|v(t)\|^{\frac{q}{m}+1}_{L^{\frac{q}{m}+1}(\Omega)}\|v\|^{\frac{1-q}{m}}_{L^\infty(Q_T)}.$$
Therefore using \eqref{30} we obtain:

$$Y'(t) \geq C\|v\|^{\frac{q-1}{m}}_{L^\infty(Q_T)}Y(t).$$

By integrating the above expression between $t=0$ and $t=T$, we have

$$Z^{m+1}(T):=\|v\|_{L^\infty(Q_T)}^{\frac{1}{m}+1}\geq Y(T)\geq  Y(0)\exp (CTZ^{q-1}(T)).$$

 The previous inequality can be rewritten as
$$ \exp\left( (m+1)Z^{1-q}(T)\ln Z(t)-Z^{1-q}(T)\ln Y(0)\right)\geq \exp(Ct).$$
Hence, we have 
$$  (m+1)Z^{1-q}(T)\ln Z(t)-Z^{1-q}(T)\ln Y(0)\geq CT$$
which gives, as $T$ goes to $\infty$, $Z(T)=\|v\|_{L^\infty(Q_T)}^\frac{1}{m}$ goes to $\infty$. 
\end{proof}

\section*{Acknowledgement}
The authors would like to thank Carlota M. Cuesta Romero for our discussions about the problem.

\appendix
\section{Properties of the operator \texorpdfstring{$A$}{A}}\label{accsec}

\begin{definition}[Accretive operator]\citep[Def 3.1, p.97]{barbu2}
A subset $\mathcal O$ of $X\times X$ (equivalently, a multivalued operator from $X$ to $X$) is called  accretive if for every pair $[x_1,y_1], [x_2,y_2] \in \mathcal O$, there is $w\in J(x_1-x_2)$ such that
$$\langle y_1-y_2,w \rangle_{X\times X^*}\geq 0,$$ 

where $J:X\to X^*$ is the the duality mapping of the space $X$ defined as in \cite[Chapter 1, p.1]{barbu2}.
The accretiveness of $\mathcal O$ is a metric geometric property
that can be equivalently expressed as 

$$\|x_1-x_2\|_X\leq \|x_1-x_2+\lambda(y_1-y_2)\|_X, \; \forall \lambda>0,\; [x_i,y_i]\in \mathcal O,\; i=1,2.$$

\end{definition}

For $X=L^1(\Omega)$ we have $J:L^1(\Omega)\to L^\infty(\Omega):$

$$J(u)=\{w \|u\|_{L^1(\Omega)}\ | \ w(x) \in sgn(u(x)) \; a.e. \},$$

where $sign$ is the multivalued sign function, for a proof see \cite[Corollary 2.7, p.59]{barbu2}.

\begin{theorem}\label{accretif}
The operator $A$ is accretive in $L^1(\Omega)$.
\end{theorem}
\begin{proof} 
The proof follows from arguments in \cite[Th. 3.5, p.117]{barbu2}. Precisely, we define $\gamma_\epsilon$ the approximation of  $sgn(\cdot)$ defined by:
\begin{align*}
\gamma_\epsilon(x)= \left\{ \begin{array}{ll} \frac{x}{|x|} &\text{ if } |x|>\epsilon, \\
\Psi_\epsilon(x) &\text{ if } |x|\leq \epsilon
\end{array} \right.
\end{align*}
where $\Psi_\epsilon\in C^1([-\epsilon,\epsilon])$, $\Psi(0)=0 , \Psi_\epsilon(\epsilon)=1$ $\Psi_\epsilon(-\epsilon)=-1$ and $\Psi_\epsilon'\geq 0$.\\
Let $u$, $v\in D(A)$. We have $\gamma_\epsilon(\pow{u}{m}-\pow{v}{m}) \in L^\infty(\Omega)\cap \Wspz)$. Indeed $\gamma_\epsilon$ is Lipschitz for any $\epsilon>0$ and
$$|\gamma_\epsilon(\pow{u}{m}-\pow{v}{m})(x)-\gamma_\epsilon(\pow{u}{m}-\pow{v}{m})(y)|\leq C |(\pow{u}{m}-\pow{v}{m})(x)-(\pow{u}{m}-\pow{v}{m})(y)|.$$
Then, $\gamma_\epsilon(\pow{u}{m}-\pow{v}{m})\in \Wspz$. Also 
$$\pow{u(x)}{m}-\pow{u(y)}{m}\geq \pow{v(x)}{m}-\pow{v(y)}{m} \implies \pow{u(x)}{m}-\pow{v(x)}{m}\geq \pow{u(y)}{m}-\pow{v(y)}{m}$$
and since $\gamma_\epsilon$ is nondecreasing,
\begin{align}\label{1}
 \begin{array}{lr}  
 \bigg(\pow{\pow{u(x)}{m}-\pow{u(y)}{m}}{p-1}-\pow{\pow{v(x)}{m}-\pow{v(y)}{m}}{p-1}\bigg) \times \\
\big(\gamma_\epsilon(\pow{u(x)}{m}-\pow{v(x)}{m})-\gamma_\epsilon(\pow{u(y)}{m}-\pow{v(y)}{m})\big) \geq 0.\end{array}
\end{align}
From \eqref{1} it follows that:
\begin{equation*}
\langle\pfrac \pow{u}{m}-\pfrac \pow{v}{m},\gamma_\epsilon(\pow{u}{m}-\pow{v}{m})\rangle\geq 0.
\end{equation*}
As in \cite{barbu2} when $\epsilon$ goes to $0$, we have $\gamma_\epsilon(\pow{u}{m}-\pow{v}{m})$ tends to $sgn(\pow{u}{m}-\pow{v}{m})=sgn(u-v)$ {\it a.e.} in $\Omega$. Since $u$, $v$ belong to $D(A)$ and from dominated convergence Theorem, we obtain:
$$\int_\Omega (\pfrac \pow{u}{m}-\pfrac \pow{v}{m})sgn(u-v)\,dx\geq 0$$
from which we deduce that the operator is accretive in $L^1(\Omega)$.
\end{proof}

\begin{proof}[Proof of Proposition \ref{propdom}]

We set $v_\epsilon\in L^1(\Omega)$ the solution of:

\begin{equation*}
    \begin{cases}
      v_\epsilon+\epsilon  A v_\epsilon=\beta(v_0) & \text{in} \;\Omega,\\
      v_\epsilon=0 & \text{in} \; \Rn\backslash \Omega.
    \end{cases}\
\end{equation*}

Indeed from Theorem \ref{exielli}, we have a solution $v_\epsilon\in D(A)$. Taking the test function $\pow{v_\epsilon}{m}-v_0\in \Wspz\cap L^\infty(\Omega)$ we have:

\begin{equation}\label{32}
\int_\Omega(v_\epsilon-\beta(v_0))(\pow{v_\epsilon}{m}-v_0)+\epsilon \langle A v_\epsilon, \pow{v_\epsilon}{m}-v_0\rangle =0.
\end{equation}

With the convexity of $v\mapsto \frac{1}{p}\|v\|_\Wspz^p$, we obtain:
\begin{equation}\label{33}
\langle A v_\epsilon, \pow{v_\epsilon}{m}-v_0\rangle \geq \frac{1}{p}\|\pow{v_\epsilon}{m}\|_\Wspz^p-\frac{1}{p} \|v_0\|_\Wspz^p.
\end{equation}
Combining \eqref{32} and \eqref{33} we have
$$\int_\Omega(v_\epsilon-\beta(v_0))(\pow{v_\epsilon}{m}-v_0)\leq\frac{ \epsilon}{p} \|v_0\|_\Wspz^p $$
and using \eqref{ineg2} and since $m>1$ we obtain:
$$\|v_\epsilon-\beta(v_0)\|^{m+1}_{L^{m+1}(\Omega)}\leq C\epsilon,$$

and then $v_\epsilon$ tends to $\beta(v_0)$ in $L^1(\Omega)$ as $\epsilon$ goes to $0$.
\end{proof}
\section{Reminders}

We recall some notion and results whom we use in this paper.

\begin{definition}[
$\epsilon$-approximate solution]\label{defapp}
Let $f\in L^1(Q_T), \epsilon>0$ and  $(t_i)_{i\in \{1,...,N\}}$ be a partition of $[0,T]$ such that for any $i$, $t_i-t_{i-1}<\epsilon$.\\
An $\epsilon$-approximate solution of \eqref{pbo} is a piecewise constant function $U:[0,T]\to L^1(\Omega)$ defined by $U(t)=U_i=U(t_{i})$ on $[t_i,t_{i+1})$ such that:
$$\|U(0)-u_0\|_{L^1(\Omega)}\leq\epsilon$$
and
$$\frac{U_i-U_{i-1}}{t_i-t_{i-1}}+\pfrac (|U_i|^{m-1} U_i)=f_i \quad \mbox{on } [t_i,t_{i+1}),$$
where $(f_i)_{i\in \{1,...,N\}}$ satisfies $ \displaystyle\sum_{i=1}^N\int_{t_{i-1}}^{t_i}\|f(\tau)-f_i\|_{L^1(\Omega)}d\tau <\epsilon$.

\end{definition}

\begin{lemma}[Gronwall lemma]\cite[Lemma 4.2.1, p. 55]{c&h}\label{gronw}
Let $T>0, \lambda\in L^1(0,T), \lambda \geq 0$ {\it a.e.} and $C_1,C_2\geq 0.$ Let $\phi \in L^1(0,T), \phi \geq 0$ {\it a.e.}, be such that $\lambda \phi\in L^1(0,T)$ and

$$\phi(t) \leq C_1+C_2\int_{0}^t\lambda(\tau)\phi(\tau)d\tau\  \mbox{ for {\it a.e.} $t\in(0,T)$}.$$
Then we have:
$$\phi(t)\leq C_1 \exp(C_2\int_{0}^t\lambda(\tau)d\tau)\ \mbox{ for {\it a.e.} $t\in(0,T)$}.$$
\end{lemma}

We have the following integration by part formula where we identify $L^2(\Omega)$ with its dual (by Riesz representation Theorem), a more general version can be found in \cite[Lemma 7.3, p.191]{roubi}.

\begin{lemma}\label{ipp}
Let $r\in (1,\infty)$. Let $V$ be a Banach space such that $V\overset{dense}{\hookrightarrow}L^2(\Omega) \hookrightarrow V^*$. Then, $W^{1,r,r'}(0,T ;V,V^*) \hookrightarrow C(0,T,L^2(\Omega))$. And for any $u,v\in W^{1,r,r'}(0,T ;V,V^*)$ and for any $t\in[0,T]$:

\begin{equation*}
\bigg[\int_\Omega uv\,dx\bigg]^t_0=\int_0^t \int_\Omega u\partial _t v+v\partial _t u \,dxd\tau,
\end{equation*}

where $[a(\tau)]^t_0=a(t)-a(0)$.

\end{lemma}


We end this section by some algebraic inequalities:

\begin{propriete}\cite{simon,kato}\label{inegalg}
There exist $c_i$, $i\in \{1,\cdots,4\}$ positive constants such that for any $\xi , \eta \in \mathbb{R}^d$:
\begin{align}\label{ineg1}\tag{\ref{inegalg}.1}
||\xi|^{p-2}\xi -|\eta|^{p-2}\eta|\leq c_1 \left\{ \begin{array}{lr}  |\xi -\eta |(|\xi|+|\eta|)^{p-2} & \text{if } p\geq 2, \\
|\xi-\eta|^{p-1} & \text{if } p \leq 2,
\end{array} \right.
\end{align}

\begin{align}\label{ineg2}\tag{\ref{inegalg}.2}
(|\xi|^{p-2}\xi -|\eta|^{p-2}\eta).(\xi-\eta) \geq c_2 \left\{ \begin{array}{lr}  |\xi -\eta |^{p} & \text{if } p\geq 2, \\
\dfrac{|\xi-\eta|^{2}}{(|\xi|+|\eta|)^{2-p}} & \text{if } p \leq 2
\end{array} \right.
\end{align}
and for any $p>1$:
\begin{equation}\label{ineg3}\tag{\ref{inegalg}.3}
c_3|\xi -\eta |(|\xi|+|\eta|)^{p-2} \leq |\pow{\xi}{p-1} -\pow{\eta}{p-1}|  \leq c_4|\xi -\eta |(|\xi|+|\eta|)^{p-2}.
\end{equation}

\end{propriete}

\bibliographystyle{plain}
\bibliography{biblio}

\end{document}